\documentclass[preprint,pdf]{imsart}

\doi{10.1214/154957804100000000}
\pubyear{0000}
\volume{0}
\firstpage{0}
\lastpage{0}
\usepackage{aliascnt}
\usepackage{hyperref}
\usepackage{graphicx}
\usepackage{url}
\usepackage{amsmath}
\usepackage[amsthm]{ntheorem}
\usepackage{amsopn,amssymb}
\usepackage{bbm,dsfont}
\usepackage{mathrsfs}
\usepackage{enumerate}
\usepackage[utf8]{inputenc}
\usepackage[textwidth=4cm, textsize=footnotesize]{todonotes}
\usepackage{color}

\startlocaldefs

\endlocaldefs
\graphicspath{{images/}}

\usepackage{cleveref}

\usepackage{shortcuts_MatComp}
\begin{document}

\begin{frontmatter}

\title{Adaptive Multinomial Matrix Completion} 
\runtitle{Adaptive Multinomial Matrix Completion}

\begin{aug}

  \author{\fnms{Olga} \snm{Klopp}\ead[label=e2]{olga.klopp@math.cnrs.fr}}
  \address{CREST and MODAL'X, Universit\'e Paris Ouest\\ \printead{e2}}

  \author{\fnms{Jean} \snm{Lafond}\ead[label=e1]{jean.lafond@telecom-paristech.fr}}
  \address{Institut Mines-T\'el\'ecom,
T\'el\'ecom ParisTech, CNRS LTCI\\ \printead{e1}}

  \author{\fnms{\'Eric} \snm{Moulines} \ead[label=e3]{moulines@telecom-paristech.fr}}
  \address{Institut Mines-T\'el\'ecom,
T\'el\'ecom ParisTech, CNRS LTCI\\ \printead{e3}}

 \author{\fnms{Joseph} \snm{Salmon}\ead[label=e4]{joseph.salmon@telecom-paristech.fr}}
  \address{Institut Mines-T\'el\'ecom,
T\'el\'ecom ParisTech, CNRS LTCI\\ \printead{e4}}

\end{aug}

\begin{abstract}
The task of estimating a matrix given a sample of observed
entries is known as the \emph{matrix completion problem}.
Most works on matrix
completion have focused on recovering an unknown real-valued low-rank
matrix from a  random sample of its entries.
Here, we investigate the case
of highly quantized observations when the measurements can take only a
small number of values. These quantized outputs are generated according
to a probability distribution parametrized by the unknown matrix of interest. 
This model corresponds, for example, to ratings in recommender
systems or labels in multi-class classification. We consider a general, non-uniform, 
sampling scheme and give theoretical guarantees on the performance
of a constrained, nuclear norm penalized maximum likelihood
estimator. One important advantage of this estimator is that it does not
require knowledge of the rank or an upper bound on the nuclear norm
of the unknown matrix and, thus, it is adaptive. We provide
lower bounds showing that our estimator is minimax optimal.  
An efficient algorithm based on lifted coordinate
gradient descent  is proposed to compute the estimator. A limited Monte-Carlo experiment, using both simulated and 
real data is provided to support our claims.
 
\end{abstract}

\begin{keyword}[class=MSC]
\kwd[Primary ]{62J02}
\kwd{62J99}
\kwd[; secondary ]{62H12,60B20}
\end{keyword}

\begin{keyword}
\kwd{Low rank matrix estimation; matrix completion; multinomial model}
\end{keyword}


\end{frontmatter}

\maketitle

\section{Introduction}

The matrix completion problem arises in a wide range of applications such as
image processing
\cite{Hui_Chaoqiang_Zuowei_Yuhong10,Ji_Musialski_Wonka_Jieping13,Xu_Jiasong_Lu_Yang_Lofti_Huazhong14},
quantum state tomography \cite{Gross11}, seismic data reconstruction \cite{Yang_Ma_Osher13} or
recommender systems \cite{Koren_Bell_Volinsky09,Bobadilla_Ortega_Hernando_Gutierrez13}.
It consists in recovering
all the entries of an unknown matrix, based on partial, random and, possibly,
noisy observations of its entries.
 Of course, since only a small proportion of
entries is observed, the problem of matrix completion is, in general, ill-posed
and requires a penalization favoring low rank solutions.
In the classical setting,
the entries are assumed to be real valued and observed in presence of
additive, homoscedastic Gaussian or sub-Gaussian noise.
In this framework, the matrix completion problem can be solved provided
that the unknown matrix is low rank, either exactly or approximately;
see \cite{Candes_Plan10,Keshavan_Montanari_Oh10,Koltchinskii_Tsybakov_Lounici11,Negahban_Wainwright12,Cai_Zhou13,Klopp14}
and the references therein.
Most commonly used methods amount to solve a least square program
under a rank constraint or its convex relaxation provided
by the nuclear (or trace) norm \cite{Fazel02}.




In this paper, we consider a statistical model where instead of observing
a real-valued entry of an unknown matrix we are now able to see only highly
quantized outputs. These discrete observations are generated according to a
probability distribution which is parameterized by the corresponding entry of the
unknown low-rank matrix.
This model is well suited to the analysis of voting patterns, preference ratings, or recovery of incomplete survey data,
 where typical survey responses are of the form  ``true/false'', ``yes/no'' or  ``agree/disagree/no opinion'' for instance.\\
The problem of matrix completion over a finite alphabet has received much less
attention than the traditional unquantized matrix completion.
One-bit matrix completion, corresponding to the case of binary, i.e. yes/no, observations, was
first introduced by \cite{Davenport_Plan_VandenBerg_Wootters12}.
In this paper, the first theoretical
guarantees on the performance of a nuclear-norm constrained maximum 
likelihood estimator are given. The sampling model considered in \cite{Davenport_Plan_VandenBerg_Wootters12}
assumes that the entries are
sampled uniformly at random. Unfortunately, this condition
is unrealistic for recommender system applications: in such a context some users are
more active than others and popular items are rated more frequently.
Another important issue is that the method of \cite{Davenport_Plan_VandenBerg_Wootters12} requires the knowledge of
an upper bound on the nuclear norm or on the rank of the unknown matrix.
Such information is usually not available in applications. On the other hand,
our estimator yields a faster rate of convergence than those obtained in \cite{Davenport_Plan_VandenBerg_Wootters12}.

One-bit matrix completion was further considered by \cite{Cai_Zhou14}
where a max-norm constrained maximum likelihood estimate is considered. This
method allows more general non-uniform sampling schemes but still requires an
upper bound on the max-norm of the unknown matrix. Here again, the rates
of convergence obtained in \cite{Cai_Zhou14} are slower than the rate of convergence of our
estimator. Recently, \cite{Gunasekar_Ravikumar_Gosh14} consider general exponential family
distributions, which cover some distributions over finite sets. Their method,
unlike our estimator, requires the knowledge of the “spikiness ratio” (usually
unknown) and the uniform sampling scheme.

In the present paper, we consider 
a maximum likelihood estimator with nuclear-norm penalization.
Our method allows us to
consider general sampling scheme and only requires the knowledge of an upper
bound on the maximum absolute value of the entries of the unknown matrix.
All the previous
works on this model also require the knowledge of this bound together with
some additional (and more difficult to obtain) information on the unknown
matrix.

The paper is organized as follows. In \Cref{subsec:bin},
the one-bit matrix completion is first discussed
 and our estimator is introduced. We establish upper
bounds both on the Frobenius norm between the unknown true matrix
and the proposed estimator
and on the associated Kullback-Leibler divergence.
In \Cref{subsec:lower} lower bounds are established,
showing that our upper bounds are minimax optimal up to logarithmic factors.
Then, the one-bit matrix completion problem is extended to the
case of a more general finite alphabet.
In \Cref{sec:implementation} an implementation based on the lifted coordinate descent
algorithm recently introduced in
\cite{Dudik_Harchaoui_Malick12} is proposed. A limited Monte Carlo experiment supporting our claims 
is then presented in \Cref{sec:numerical-experiment}.

\subsection*{Notations}
For any integers $n,m_1,m_2>0$, $[n] \eqdef \{1,\dots,n\}$,
 $m_1 \vee m_2 \eqdef \max(m_1,m_2)$ and
$m_1 \wedge m_2 \eqdef \min(m_1,m_2)$.
We equip the set of $m_1 \times m_2$ matrices with real entries
(denoted $\matset{m_1}{m_2}$) with the scalar product
$\langle \mat{X}|\mat{X'} \rangle \eqdef \tr(\mat{X}^\top \mat{X'})$.
 For a given matrix $\mat{X}\in \matset{m_1}{m_2}$ we write
$\|\mat{X}\|_\infty \eqdef \max_{i,j} |\mat{X}_{i,j}|$ and for any $\rho \geq 1$, we denote its Schatten $\rho$-norm
(see \cite{Bhatia97}) by
 \begin{equation*}
  \|\mat{X}\|_{\sigma,\rho}\eqdef \left( \sum_{i=1}^{m_1\wedge m_2} \sigma^\rho_i(\mat{X}) \right)^{1/\rho} \eqs,
 \end{equation*}
with $\sigma_i(\mat{X})$ the singular values of $\mat{X}$ ordered in decreasing order.
The operator norm of $\mat{X}$ is $\|\mat{X}\|_{\sigma,\infty}\eqdef\sigma_1(\mat{X})$.
For any integer $\qClass >0$, we denote by $\tenset{m_1}{m_2}{\qClass}$ the set of $m_1 \times m_2 \times \qClass$
(3-way) tensors. A tensor $\ten{X}$ is of the form $\ten{X}= (\mat{X}^\qclass)_{\qclass=1}^{\qClass}$
where $\mat{X}^\qclass \in \RR^{m_1 \times m_2}$ for any $\qclass \in [\qClass]$.
For any integer $\Class>0$, a function $f:\RR^{\qClass} \to \mathcal{S}_{\Class}$ is called a $\Class$-link
function, where $\mathcal{S}_{\Class}$ is the $\Class-$dimensional probability simplex.
 Given a $\Class$-link function $f$ and
$\ten{X},\ten{X'} \in \tenset{m_1}{m_2}{\qClass}$, we define the squared Hellinger distance

\begin{equation*}
\dhes{f(\ten{X})}{f(\ten{X'})} \eqdef
\frac{1}{m_1m_2} \sum_{k \in [m_1] }  \sum_{k' \in [m_2] }
\sum_{\class \in [\Class]} \left[ \left( \sqrt{f^\class(\ten{X}_{k,k'})}-\sqrt{f^\class(\ten{X}'_{k,k'})} \right)^2
\right] \eqsp ,
\end{equation*}
where $\ten{X}_{k,k'}$ denotes the vector $(\mat{X}^\class_{k,k'})_{\class=1}^{\qClass}$.
The Kullback-Leibler divergence is
\begin{equation*}
\KL{f(\ten{X})}{f(\ten{X'})} \eqdef
\frac{1}{m_1m_2} \sum_{k \in [m_1] }  \sum_{k' \in [m_2] }\sum_{\class \in [\Class]}
\left[f^\class(\ten{X}_{k,k'})\log \left(\frac{f^\class(\ten{X}_{k,k'})}{f^\class(\ten{X}'_{k,k'}) } \right) \right] \eqsp.
\end{equation*}
For any tensor $\ten{X} \in \tenset{m_1}{m_2}{\qClass}$ we define
$\rank(\ten{X})\eqdef\max_{\qclass\in [\qClass]} \rank(\mat{X}^\qclass)$, where $\rank(\mat{X}^\qclass)$ is the
rank of the matrix $\mat{X}^\qclass$ and its
sup-norm by $\|\ten{X}\|_\infty \eqdef \max_{\qclass \in [\qClass]} \|\mat{X}^\qclass \|_\infty$.

\section{Main results}
\label{sec:main-results}
\subsection{One-bit matrix completion}
\label{subsec:bin}
Assume that the observations follow a binomial distribution
 parametrized by a matrix $\mat{\tX} \in \matset{m_1}{m_2}$.
Assume in addition that an $\iid$ sequence of coefficients $(\omega_i)_{i=1}^n \in ([m_1] \times [m_2])^n$
is revealed and denote by $\Pi$ their distribution.
The observations associated to these coefficients are denoted by $(Y_i)_{i=1}^n\in \{1,2\}^n$
and distributed as follows
\begin{equation}
\label{eq:definition-f}
  \PP(Y_i=\class)=f^\class(\mat{\tX}_{\omega_{i}}), \quad \class \in \{1,2\} \eqsp,
\end{equation}
where $f=(f^\class)_{\class=1}^2$ is a $2-$link function.
For ease of notation, we often write $\mat{\tX}_i$  instead of  $\mat{\tX}_{\omega_i}$.
Denote by $\Lik$ the (normalized) negative log-likelihood of the observations:
\begin{equation}
\label{eq:likelihood-binomial1C}
   \Lik(\mat{X}) = -\frac{1}{n} \sum_{i=1}^{n}\left(
   \sum_{\class=1}^{2} \1_{\{Y_i=\class\}}\log\left(f^\class(\mat{X}_i)\right) \right) \eqsp.
\end{equation}

Let $\gamma>0$ be an upper bound of $\| \mat{\tX} \|_{\infty}$. We consider the following estimator
of $\tX$:
\begin{equation}
\label{eq:MinPb1C}
 \hat{\mat{X}}=\argmin_{\substack{\mat{X} \in \RR^{m_1 \times m_2}, \|\mat{X}\|_{\infty}\leq \gamma }}
\Obj(\mat{X}) \eqsp, \quad \text{where} \quad  \Obj(\mat{X})= \Lik(\mat{X}) +
 \lambda  \|\mat{X} \|_{\sigma,1} \eqsp,
\end{equation}
with $\lambda>0$ being a regularization parameter.
Consider the following assumptions.
\begin{assumption}
\label{A0}
The functions $x \mapsto - \ln(f^j(x))$, $j=1,2$ are convex. In addition,
There exist positive constants $H_\gamma$, $L_\gamma$ and $K_\gamma$ such that:
\begin{align}
\label{eq:definition:M-gamma-bin}
H_\gamma \geq & 2 \sup_{|x|\leq\gamma}(|\log(f^1(x))| \vee |\log(f^2(x))|) \eqsp,\\
\label{eq:definition:L-gamma-bin}
L_\gamma \geq &\max\left(\sup_{|x|\leq\gamma} \frac{|(f^1)'(x)|}{f^1(x)},\sup_{|x|\leq\gamma}  \frac{|(f^2)'(x)|}{f^2(x)}\right) \eqsp,\\
 \label{eq:definition:K-gamma-bin}
 K_\gamma =  &\inf_{|x|\leq \gamma}g(x) \eqsp, \quad \text{	where } g(x)= \frac{(f^1)'(x)^2}{8f^1(x)(1-f^1(x))} \eqsp.
 \end{align}
\end{assumption}

 \begin{remark}
 \label{rem:K_gamma}
As shown in \cite[Lemma~2]{Davenport_Plan_VandenBerg_Wootters12}, $K_\gamma$ satisfies
 \begin{equation}
 \label{eq:minoration-utile}
 K_\gamma \leq \inf_{\substack{x,y\in \RR\\
  |x| \leq \gamma \\ |y| \leq \gamma}} \left( \sum_{\class=1}^2 \left(\sqrt{f^\class(x)}-\sqrt{f^\class(y)} \right)^2/(x-y)^2 \right) \eqsp.
 \end{equation}
\end{remark}

Our framework allows a general distribution $\Pi$. We assume that $\Pi$ satisfies the following assumptions introduced in \cite{Klopp14} in the classical setting of unquantized matrix completion:
\begin{assumption}
\label{A1} There exists a constant $\mu > 0$ such that, for any $m_1 > 0$ and $m_2 > 0$
\begin{equation}
\label{eq:definition-mu}
\min_{k \in [m_1], k' \in [m_2]} \pi_{k,k'} \geq \mu / (m_1 m_2) \eqsp, \quad \text{where  } \pi_{k,k'}= \PP( \omega_1= (k,k'))  \eqsp.
\end{equation}
\end{assumption}
Denote by $R_k=\sum_{k'=1}^{m_2} \pi_{k,k'}$ and $C_{k'}=\sum_{k=1}^{m_1} \pi_{k,k'}$  the probability of revealing a coefficient from row $k$ and column $k'$, respectively.
\begin{assumption}
\label{A2}
There exists a constant $\Lc \geq 1$ such that, for all $m_1, m_2$,
\begin{equation*}
 \max_{k,l}(R_k,C_l) \leq \frac{\Lc}{m_1 \wedge m_2} \eqsp,
\end{equation*}
\end{assumption}
The first assumption ensures that every coefficient has a nonzero probability of
being observed, whereas the second assumption requires that no column
nor row is sampled with too high probability
(see also \cite{Foygel_Salakhutdinov_Shamir_Srebro11,Klopp14} for more details on these conditions).
For instance, the uniform distribution yields $\mu=\nu=1$.
Define
\begin{equation}
\label{eq:definition-d-M}
d= m_1+ m_2 \eqsp, \quad M= m_1 \vee m_2, \quad m= m_1 \wedge m_2  \eqsp.
\end{equation}

\begin{theorem}
\label{col:bin}
Assume \Cref{A0}, \Cref{A1}, \Cref{A2} and that  $\| \tX \|_\infty \leq \gamma$. Assume in addition that $n \geq 2 m\log(d)/(9\Lc)$.
Take
 \begin{equation*}
  \lambda = 6L_\gamma\sqrt{ \frac{2 \Lc \log(d)}{m n}} \eqsp.
 \end{equation*}
 Then, with probability at least  $1-3 d^{-1}$
 the Kullback-Leibler divergence
is bounded by
  \begin{equation*}
  \KL{f(\mat{\tX})}{f(\hat{\mat{X}})} \leq \mu \max  \left( \bar{c} \mu \nu\frac{L_\gamma^2 \rank(\bar{\mat{X}})}{K_\gamma} \frac{ M \log(d)}{n} ,
\rme H_\gamma\sqrt{\frac{\log(d)}{n}} \right)\eqsp,
 \end{equation*}
with $\bar{c}$ a universal constant whose value is specified in the proof. 

\end{theorem}
\begin{proof}
 See \Cref{proofth1}.
\end{proof}
This result immediately gives an upper bound on the estimation error of $\hat X$, measured in  Frobenius norm:
\begin{corollary}
\label{col1}
Under the same assumptions and notations of \Cref{col:bin} we have with probability
at least $1-3d^{-1}$
  \begin{equation*}
  \frac{\|\mat{\tX}-\hat{\mat{X}}\|^2_{\sigma,2}}{m_1m_2} \leq \mu \max  \left( \bar{c} \mu \nu\frac{L_\gamma^2 \rank(\bar{\mat{X}})}{K^2_\gamma} \frac{ M \log(d)}{n} ,
\frac{\rme H_\gamma}{K_\gamma} \sqrt{\frac{\log(d)}{n}} \right)\eqsp .
 \end{equation*}
\end{corollary}
\begin{proof}
Using \Cref{HelFro} and \Cref{col:bin}, the result follows.
\end{proof}

\begin{remark}
Note that, up to the factor $L_\gamma^2/K^{2}_\gamma$, the rate of convergence given by   \Cref{col1},  
is the same as in the case of usual unquantized matrix completion, see, for example,
\cite{Klopp14} and \cite{Koltchinskii_Tsybakov_Lounici11}.
 For this usual  matrix completion setting, it has been shown in \cite[Theorem 3]{Koltchinskii_Tsybakov_Lounici11} that
this rate is minimax optimal up to a logarithmic factor.
Let us compare this rate of convergence with those obtained in previous works on 1-bit matrix completion.
 In \cite{Davenport_Plan_VandenBerg_Wootters12}, the parameter
  $\bar X$ is estimated by minimizing the negative log-likelihood
  under the constraints $\|\mat{X}\|_\infty \leq \gamma$ and $\|\mat{X}\|_{\sigma,1} \leq  \gamma
 \sqrt{rm_1m_2}$ for some $r>0$.
  Under the assumption that $\rank(\mat{\tX}) \leq r$, they could prove that
  \begin{equation*}
   \frac{\|\mat{\tX}-\hat{\mat{X}}\|^2_{\sigma,2}}{m_1m_2} \leq C_\gamma \sqrt{\frac{rd}{n}} \eqsp,
  \end{equation*}
 where $C_\gamma$ is a constant depending on $\gamma$ (see \cite[Theorem 1]{Davenport_Plan_VandenBerg_Wootters12}).
 This rate of convergence is slower than the rate of convergence given by \Cref{col1}. \cite{Cai_Zhou14} studied a max-norm
 constrained maximum likelihood estimate and obtain a rate of convergence similar to \cite{Davenport_Plan_VandenBerg_Wootters12}.
 In \cite{Gunasekar_Ravikumar_Gosh14}, matrix completion was considered for
 a likelihood belonging to the exponential family. Note, for instance, that the logit distribution
 belongs to such a family. The following upper bound on the estimation error is provided (see \cite[Theorem 1]{Gunasekar_Ravikumar_Gosh14})
   \begin{align}\label{olga_1}
   \frac{\|\tX-\hat{{X}}\|^2_{\sigma,2}}{m_1m_2}
 &{\leq C_\gamma}\left(\alpha_*^2 \frac{\rank(\tX) M \log(M) }{n}\right).
  \end{align}
 Comparing with \Cref{col1}, \eqref{olga_1} contains an additional term $\alpha_*^2$ where
 $\alpha_*$ is an upper bound of $\sqrt{m_1m_2} \|\tX\|_\infty$.
\end{remark}

\subsection{Minimax lower bounds for one-bit matrix completion}
\label{subsec:lower}
\Cref{col1} insures that our estimator achieves certain Frobenius norm errors.
We now discuss the extent to which this result is optimal. A classical way to address this question is by
determining minimax rates of convergence.

 For any integer $0\leq r\le
\min(m_1,m_2)$ and any $\gamma>0$, we consider the following family of matrices
\begin{equation*} 
 \begin{split}
 \mathcal{ F}(r,\gamma)
 &= \left \{\mat{\tX}\in\,\mathbb R^{m_1\times m_2}:\,
 \mathrm{rank}(\mat{\tX})\leq r,\,\Vert\mat{\tX}\Vert_{\infty}\leq \gamma \right \}.
 \end{split}
 \end{equation*}
We will denote by $\inf_{\mat{\hat{X}}}$ the infimum over all  estimators
$\mat{\hat{X}}$  that are functions of the data $(\omega_i,Y_i)_{i=1}^{n}$.
For any $\mat{X}\in\matset{m_1}{m_2}$, let $\PP_{\mat{X}}$ denote the probability
distribution of the observations $(\omega_i,Y_i)_{i=1}^{n}$
 for a given $2-$link function $f$ and sampling distribution $\Pi$.
We establish a lower bound under an additional assumption on the function $f^1$:
\begin{assumption}
\label{h:minimax}
$(f^1)'$ is decreasing on $\RR_+$ and $K_\gamma=g(\gamma)$ where $g$ and $K_\gamma$
are defined in \eqref{eq:definition:K-gamma-bin}.
\end{assumption}
In particular, \Cref{h:minimax} is satisfied in the case of logit or probit models.
The following theorem establishes a lower bound on the minimax risk in squared Frobenius norm:
%

\begin{theorem}\label{th:lower}
 Assume \Cref{h:minimax}. Let $\alpha \in(0,1/8)$
Then there exists a constant $c>0$ such that,
 for all $m_1,m_2\geq 2$, $1\leq  r\leq m$,
and $\gamma>0$,
\begin{equation*}
\inf_{\mat{\hat{X}}}
\sup_{\substack{\mat{\tX}\in\,{\cal F}( r,\gamma)
}}
\PP_{\mat{\tX}}\left (\dfrac{\Vert \mat{\hat{X}}-\mat{\tX}\Vert_2^{2}}{m_1m_2} > c\min\left \{\gamma^{2}, \dfrac{M r}{n\,K_{0}}\right \} \right )\ \geq\ \delta(\alpha,M) \eqsp,
\end{equation*}
where
\begin{equation}
\label{eq;definition-delta}
\delta(\alpha,M)=\frac{1}{1+2^{-rM/16}}\left(1-2\alpha-\frac{1}{2} \sqrt{\frac{\alpha}{\log(2) (rM)}} \right) \eqsp.
\end{equation}
\end{theorem}
\begin{proof}
 See \Cref{proof:thlower}.
\end{proof}
Note that the lower bound given by \Cref{th:lower} is proportional to the rank
multiplied by the maximum dimension of $\tX$ and inversely proportional the sample size $n$.
Therefore the lower bound matches the upper bound given by \Cref{col1} up to a constant and a logarithmic
factor.
The lower bound does not capture the dependance on $\gamma$, note however that the upper and lower bound only
differ by a factor $ L^{2}_\gamma\,/K_{\gamma}$.

\subsection{Extension to multi-class problems}
\label{subsec:gen}
Let us now consider a more general setting where the observations follow a distribution
over a finite set $\{1,\dots,\Class\}$,
 parameterized by a tensor $\ten{\tX} \in \tenset{m_1}{m_2}{\qClass}$.
The distribution of the observations $(Y_i)_{i=1}^n\in [\Class]^n$ is
\begin{equation*}
  \PP(Y_i=\class)=f^\class(\ten{\tX}_{\omega_{i}}), \quad \class \in [\Class] \eqsp,
\end{equation*}
where $f=(f^\class)_{\class=1}^\Class$ is now a $\Class$-link function
and $\ten{\tX}_{\omega_{i}}$ denotes the vector $(\mat{\tX}_{\omega_i}^\qclass)_{\qclass=1}^{\qClass}$.
The negative log-likelihood of the observations is now given by:
\begin{equation}
\label{eq:likelihood-binomial}
   \Lik(\ten{X}) = -\frac{1}{n} \sum_{i=1}^{n}\left(
   \sum_{\class=1}^{\Class} \1_{\{Y_i=\class\}}\log\left(f^\class(\ten{X}_i)\right) \right) \eqsp.
\end{equation}
where we use the notation $\ten{X}_i=\ten{X}_{\omega_i}$.
Our proposed the estimator is defined as:
\begin{equation}
\label{eq:MinPb}
 \hat{\ten{X}}=\argmin_{\substack{\ten{X} \in \tenset{m_1}{m_2}{\qClass} \\ \|\ten{X}\|_{\infty}\leq \gamma }} \Obj(\ten{X}) \eqsp, \quad \text{where} \quad  \Obj(\ten{X})= \Lik(\ten{X}) +
 \lambda \sum \limits_{\class=1}^{\qClass} \|\mat{X}^\class \|_{\sigma,1} \eqsp,
\end{equation}

In order to extend the results of the previous sections we make an additional assumption
which allows to split the log-likelihood as a sum.

\begin{assumption}
\label{A3}
There exist functions $(g_\qclass^\class)_{(\qclass,\class)\in [\Class]\times[\qClass]}$
such that the $\Class$-link function $f$ can be factorized as follows
\begin{equation*}
 f^\class(x_1,\dots,x_\qClass)=\prod_{\qclass=1}^{\qClass} g^\class_\qclass(x_\qclass) \quad \text{for } \class \in [\Class] \eqsp.
\end{equation*}
\end{assumption}
The model considered above covers many finite distributions
including among others logistic binomial (see \Cref{subsec:bin}) and
conditional logistic multinomial (see \Cref{sec:implementation}).

Assumptions on constants depending on the link function are extended by
\begin{assumption}
 \label{A0m}
There exist positive constant $H_\gamma$, $L_\gamma$ and $K_\gamma$ such that:
\begin{align}
\label{eq:definition:M-gamma}
H_\gamma \geq &\max_{(\class,\qclass) \in [\Class]\times [\qClass]} \sup_{|x|\leq\gamma}2 |\log(g^\class_\qclass(x))| \eqsp,\\
\label{eq:definition:L-gamma}
\quad L_\gamma \geq &\max_{(\class,\qclass) \in [\Class]\times [\qClass]}\sup_{|x|\leq\gamma} \left| \frac{(g_\qclass^j)'(x)}{g_\qclass^j(x)} \right|\eqsp,\\
\label{eq:definition:K-gamma}
\quad K_\gamma \leq &\inf_{\substack{x,y\in \RR^\qClass \\
 \|x\|_\infty \leq \gamma \\ \|y\|_\infty \leq \gamma}} \left( \sum_{\class=1}^\Class \left(\sqrt{f^\class(x)}-\sqrt{f^\class(y)} \right)^2/\|x-y\|_2^2 \right) \eqsp.
\end{align}
\end{assumption}
For any tensor $\ten{X} \in \tenset{m_1}{m_2}{\qClass}$, we write
$\bar{\Sigma}\eqdef \nabla \Lik(\ten{\tX}) \in \tenset{m_1}{m_2}{\qClass}$.
We also define the sequence of matrices $(E_i)_{i=1}^{n}$ associated to the revealed coefficients $(\omega_i)_{i=1}^{n}$ by
$E_i\eqdef~e_{k_i} (e'_{l_i})^\top$ where $(k_i,l_i)=\omega_i$ and with
$(e_k)_{k=1}^{m_1}$ (\resp\ ($(e'_l)_{l=1}^{ m_2}$) being the canonical basis of
$\RR^{m_1}$ (\resp\ $\RR^{m_2}$).
Furthermore, if $(\varepsilon_i)_{1\leq i \leq n}$ is a Rademacher sequence independent
 from  $(\omega_i)_{i=1}^{n}$ and $(Y_i)_{1\leq i \leq n}$ we define the matrix $\Sigma_R$ as follow
\begin{equation*}
 \Sigma_R\eqdef\frac{1}{n}\sum_{i=1}^{n} \varepsilon_i E_i \eqsp.
\end{equation*}

We can now state the main results of this paper.
\begin{theorem}
\label{th1}
 Assume  \Cref{A1}, \Cref{A3} and \Cref{A0m} hold, $\lambda>2\max_{\qclass \in [\qClass]} \|\bar{\Sigma}^\qclass \|_{\sigma,\infty}$
 and \\
 $\norm{\ten{\tX}}_\infty \leq \gamma$.
 Then, with probability at least $1-2d^{-1}$,
 the Kullback-Leibler divergence
is bounded by
 \begin{multline*}
  \KL{f(\ten{\tX})}{f(\ten{\hat X})} \leq \\
  \mu \max \left(4 \mu \frac{m_1m_2 \mrk{\ten{\tX}} }{K_\gamma}
  \left( \lambda^2 +  256 \rme (\qClass L_\gamma \PE \norm{\Sigma_R}_{\sigma,\infty})^2\right)
  ,\rme H_\gamma\sqrt{\frac{\log(d)}{n}} \right)\eqsp.
 \end{multline*}
with $d$ defined in \eqref{eq:definition-d-M}.

\end{theorem}
\begin{proof}
 See \Cref{proofth1}.
\end{proof}

Note that the lower bound of $\lambda$ is stochastic and the expectation $\EE \|\Sigma_R \|_{\sigma,\infty}$
is unknown. However, these quantities can be controlled using \Cref{A2}.

\begin{theorem}
\label{th2}
Assume \Cref{A1}, \Cref{A2}, \Cref{A3} and \Cref{A0m} hold
and that $\| \ten{\tX} \|_\infty \leq \gamma$. Assume in addition that $n \geq 2 m\log(d)/(9\Lc)$.
Take
 \begin{equation*}
  \lambda = 6 L_\gamma\sqrt{ \frac{2 \Lc \log(d)}{mn}} \eqsp.
 \end{equation*}
 Then, with probability at least  $1-(2+\qClass)d^{-1}$,
 the Kullback-Leibler divergence
is bounded by
 \begin{multline*}
\KL{f(\ten{\tX})}{f(\ten{\hat X})} \leq
\mu \max  \left( \bar{c} \mu \nu\frac{\qClass^2L_\gamma^2 \mrk{\ten{\tX}} }{K_\gamma}
\frac{ M\log(d)}{n} ,
\rme H_\gamma\sqrt{\frac{\log(d)}{n}} \right)\eqsp,
 \end{multline*}
with $\bar{c}$ a universal constant
, $d$, $m$ and $M$ defined in \eqref{eq:definition-d-M}. 
\end{theorem}
\begin{proof}
 See \Cref{proofth2}.
\end{proof}

\section{Implementation}
\label{sec:implementation}

In this section an implementation for the following $\Class$-class link function is given:
\begin{align*}
f^\class(x^1,\dots,x^{\Class-1})=
\begin{cases}
\exp({x}^\class) \left(\displaystyle\prod_{\qclass=1}^{\class}(1+\exp({x^\qclass}))\right)^{-1} & \text{if }  \class \in [\Class-1] \eqs,\\
\left(\displaystyle\prod_{\qclass=1}^{\Class-1}(1+\exp({x^\qclass}))\right)^{-1} & \text{if }  \class=\Class \eqs.
\end{cases}
\end{align*}
This $\Class$-class link function boils down to parameterizing the distribution of the observation as follows:
\begin{align*}
&\PP(Y_i=1)=\frac{\exp(\mat{\tX_i}^1)}{1+\exp(\mat{\tX_i^1})} \eqsp, \\
&\PP(Y_i=\class|Y_i>\class-1)=\frac{\exp(\mat{\tX_i}^\class)}{1+\exp(\mat{\tX_i^\class})} \quad \text{for } \class \in \{2,\dots,\Class-1\} \eqs.
\end{align*}
Assumption \Cref{A3} is satisfied and the problem \eqref{eq:MinPb} is separable \wrt each matrix $\mat{X^\qclass}$.
Following \cite{Davenport_Plan_VandenBerg_Wootters12}, we solve \eqref{eq:MinPb} without taking into account the constraint $\gamma$; as reported in \cite{Davenport_Plan_VandenBerg_Wootters12} and confirmed by our experiments, the impact of this projection is negligible, whereas it increases
significantly the computation burden.

Because the problem is separable, it suffices to solve in parallel each sub-problem
\begin{equation}
\label{eq:MinSubPb}
 \hat{\mat{X}^\qclass}=\argmin_{\substack{\mat{X} \in \matset{m_1}{m_2}}}
 \Objo^{\qclass}_\lambda(\mat{X}) \eqsp, \quad \text{where} \quad  \Objo^{\qclass}_\lambda(\mat{X})= \Liko^\qclass(\mat{X}) +
 \lambda \|\mat{X} \|_{\sigma,1} \eqsp.
\end{equation}
This can be achieved by using the coordinate gradient
descent algorithm introduced by \cite{Dudik_Harchaoui_Malick12}.
To describe the algorithm, consider first the set of normalized rank one matrices
\begin{equation*}
 \tens:=\left\{ M \in \RR ^{m_1 \times m_2} |
 M=  uv^\top ~|~ \|u \|_2=\|v \|_2=1, \ \right\} \,.
\end{equation*}
Define  $\Theta$ the linear space of real-valued functions on $\tens$ with finite support, 
\ie\ any $\theta \in \Theta$ satisfies $\theta(M)= 0$ except for a finite number of $M \in \tens$.
This space is equipped with the $\ell^1$-norm  $\| \theta \|_1= \sum_{M \in \tens} |\theta(M)|$.
Define by $\Theta_+$ the positive orthant, \ie\ the cone of functions $\theta \in \Theta$ such that $\theta(M) \geq 0$ for all $M \in \tens$.
Any matrix $\mat{X}\in\matset{m_1}{m_2}$ can be associated to an element $\theta\in \Theta_+$ satisfying
\begin{equation}
\label{eq:representation-canonique}
 \mat{X}= \sum_{M\in \tens} \theta(M) M\eqs .
\end{equation}
Such function is not unique. Consider
an SVD of $\mat{X}$ \ie $\mat{X}=\sum_{i=1}^{m} \lambda_i u_i v_i^\top$, 
where $(\lambda_i)_{i=1}^m$ are the singular values and $(u_i)_{i=1}^m$, $(v_i)_{i=1}^m$ are left and right singular vectors,
then $\theta_{\mat{X}} =\sum_{i=1}^{m} \lambda_i \delta_{u_i v_i^\top}$ satisfies \eqref{eq:representation-canonique},
with 
$\delta_M \in \Theta$ is the function on $\tens$ satisfying $\delta_M(M)=1$ and $\delta_M(M')=0$ if $M' \neq M$.
As seen below, the function $\theta_{\mat{X}}$ plays a key role.

Conversely, for any  $\theta \in \Theta_+$,  define
\begin{equation*}
 \wrec: \theta \to \wrec_\theta :=\sum_{M \in \tens} \theta(M) M \eqs .
\end{equation*}
and the auxiliary objective function:
\begin{equation}
\label{eq:definition-Objb}
\Objob^\qclass: \theta \to \Objob^\qclass(\theta)\eqdef \lambda  \sum_{M \in \tens} \theta(M) + \Liko^\qclass(\wrec_\theta) \eqs.
\end{equation}
The triangular inequality implies that for all $\theta \in \Theta_+$,
\[
 \|\mat{\wrec_\theta}\|_{\sigma,1} \leq \| \theta \|_1 \eqs.
\]
For $\theta \in \Theta$ we denote by $\supp(\theta)$ the support of $\theta$
\ie the subset of $\tens$ such that $\theta(M)\neq 0 \iff \mat{M} \in \supp(\theta)$.
If for any $\mat{M},\mat{M}' \in \supp(\theta)$, $\mat{M}\neq \mat{M}'$, $\pscal{\mat{M}}{\mat{M}'}=1$
, then
$\|\theta\|_1= \| \mat{\wrec_\theta} \|_{\sigma,1}$. Indeed in such case $\sum_{M \in \tens} \theta(M) M$
defines a SVD of $\mat{\wrec_\theta}$.
Therefore the minimization of \eqref{eq:definition-Objb} is actually equivalent to the minimization of
\eqref{eq:MinSubPb}; see \cite[Theorem~3.2]{Dudik_Harchaoui_Malick12}. The minimization
 \eqref{eq:definition-Objb} can be implemented using a
coordinate gradient descent algorithm which updates at each iteration the nonnegative finite support function $\theta$.

\begin{algorithm}[h]
\label{alg:MCLGD}
\textbf{Initialization:} initial parameter $\theta_0$, precision $\epsilon$ \\
\textbf{Loop:}\\

Compute the top singular vector pair of $-\nabla \Liko^\qclass(\mat{W_{\theta_k}})$: $u_k$, $v_k$\\

$g_k \leftarrow \lambda + \pscal{\nabla\Liko^\qclass(\mat{W_{\theta_k}})}{u_k v_k ^\top}$\\
\If{$g_k \leq -\epsilon/2$}{
  $\beta_k \leftarrow \displaystyle \argmin_{b \in \RR_+}
   \Objob^\qclass\left(\theta+ b \delta_{u_k v_k^\top})\right)$\\
  $\theta_{k+1} \leftarrow \theta_k+ \beta_k \delta_{u_k v_k^\top}$, \\
}
\Else{
  $g_k^{\rm max} \leftarrow  \displaystyle \max_{M \in \supp(\theta_k)}
  |\lambda + \pscal{\nabla\Liko^\qclass(W_{\theta_k})}{M}|$ \\
  \If{$g_k^{\rm max} \leq \epsilon$}{
  \, Break
  }\Else{
  $\theta_{k+1} \leftarrow \displaystyle \argmin_{\theta' \in \Theta_+,
  \supp(\theta') \subset \supp(\theta_k)} \Objob^\qclass(\theta')$\\
  }
}

\caption{Lifted coordinate gradient descent}
\end{algorithm}

The algorithm is summarized in Algorithm~\ref{alg:MCLGD}.
Compared to the Soft-Impute \cite{Mazumder_Hastie_Tibshirani10}
or the SVT \cite{Cai_Candes_Shen10} algorithms, this algorithm does not require the computation of a full SVD at each step
of the main loop of an iterative (proximal) algorithm (recall that the proximal operator
associated to the nuclear norm is the soft-thresholding operator of the singular values).
The proposed algorithm requires only to compute the largest singular values and associated singular vectors.

Another interest of this algorithm is that it only requires to evaluate the coordinate of the gradient for the entries which have been actually observed. It is therefore memory efficient when the number of observations is smaller than the total number of coefficients $m_1 m_2$, which is the typical setting in which matrix completion is used.
Moreover, we use Arnoldi
iterations to compute the top singular values and vector
pairs (see \cite[Section 10.5]{Golub_VanLoan13} for instance) which
allows us to take full advantage of sparse structures, the minimizations
in the inner loop are carried out using the L-BFGS-B algorithm.
\Cref{tab:exec} provides the execution time one-bit matrix completion (on a 3.07Ghz w3550 Xeon CPU with RAM~1.66 Go, Cache~8 Mo, C implementation).

\begin{table}[t!]
\centering
\begin{tabular}{|l|c|c|c|}
   \hline
   \textbf{Parameter Size} & $1000\times1000$ &$3000\times3000$ &$10000\times10000$   \\
   \hline
   \textbf{Observations} & $100\cdot 10^3$ &$1\cdot 10^6$ & $10\cdot 10^6$   \\
   \hline
   \textbf{Execution Time (s.)}& $4.5$ & $52$ & $730$ \\
   \hline
\end{tabular}
\caption{Execution time of the proposed algorithm for the binary case.
}
\label{tab:exec}
\end{table}

\section{Numerical Experiments}

\label{sec:numerical-experiment}
\label{subsec:condition_multi}

We have performed numerical experiments on both simulated and real data provided by the  MovieLens project (\url{http://grouplens.org}).
Both the one-bit matrix completion - $\Class=2$, $\qClass=1$ - and the extended multi-class setting -$\Class=5$, $\qClass=4$ -
are considered; comparisons are also provided with the classical Gaussian matrix completion algorithm to assess the potential
gain achieved by explicitly taking into account the facts that the observations belong to a finite alphabet.
Only a limited part of the experiments are reported in this article; a more extensive assessment can be obtained upon authors request.

For each matrix $\mat{\tX}^\qclass$ we sampled uniformly five unitary (for the Euclidean norm)
vector pairs $(u_k^\qclass, v_k^\qclass)_{k=1}^{5}$. The matrix $\mat{\tX}^\qclass$ is then defined
as
\begin{equation*}
\mat{\tX}^\qclass=\Gamma\sqrt{m_1m_2} \sum \limits_{k=1}^{5}  \alpha_k u_k^\qclass( v_k^\qclass)^\top + \eta ^\qclass\mathbf{I}_{m_1\times m_2},
\end{equation*}
with $(\alpha_1,\dots,\alpha_5)=(2,1,0.5,0.25,0.1)$,
$\Gamma$ a scaling factor and $\mathbf{I}_{m_1\times m_2}$ the $m_1~\times~m_2$ matrix of ones.
The term $\eta_l$ has been fixed so that each class has the same average probability
\ie $f^\class((\EE[\mat{\tX}^\qclass])_{\qclass=1}^{\Class-1})=1/\Class$ for $\class \in [\Class]$.
Note that the factor $\sqrt{m_1m_2}$
implies that the variance of $\mat{\tX}^\qclass$ coefficients does not depend
on $m_1$ and $m_2$.
The sizes investigated are $(m_1,m_2) \in \{(500, 300),(1000,600)\}$.

The observations are sampled to the conditional multinomial logistic model introduced in  \Cref{sec:implementation}.
For comparison purposes we have also computed $\hat{\mat{X}}^\mathcal{N}$,
the classical Gaussian version (\ie using a squared Frobenius norm
in \eqref{eq:MinPb}).
Contrary to the logit version, the Gaussian matrix completion does not directly recover the distribution of
the observations $(Y_i)_{i=1}^n$.
However, we can estimate  ${\PP}(Y_i=\class)$ by the following quantity:
\begin{equation*}
F_{\mathcal{N}(0,1)}(p_{\class+1})-F_{\mathcal{N}(0,1)}(p_\class)
 \text{  with  } p_\class =
 \begin{cases}
 0 &\text{ if } \class =1\eqs,\\
\frac{\class-0.5-\hat{\mat{X}}^\mathcal{N}_{i}}{\hat{\sigma}} &\text{ if } 0<\class<\Class \\
1  &\text{ if }\class = \Class \eqs,
\end{cases}
\end{equation*}
where $F_{\mathcal{N}(0,1)}$ is the cdf of a zero-mean standard Gaussian random variable.

The choice of the regularization parameter $\lambda$ has been solved for
all methods by performing 5-fold cross-validation on a geometric grid of size $0.6 \log(n)$
(note that the estimators are null for $\lambda$ greater than $\| \nabla \Lik(0)\|_{\sigma, \infty}$).

As evidenced in \Cref{fig:KL},
the Kullback-Leibler divergence for the
logistic estimator is significantly lower than for the Gaussian estimator, for both the  $\Class=2$ and $\Class=5$ cases.
This was expected because the Gaussian
model assume implicitly symmetric distributions with the same variance for all the ratings,
These assumptions are of course avoided by the  logistic modem.

Regarding the prediction error, \Cref{tab:sim_bin_error} and \Cref{tab:sim_mult_error} summarize
the results obtained for a $1000 \times 600$ matrix.
The logistic model outperforms the Gaussian model (slightly for $\Class=2$ and significantly for  $\Class=5$).


\begin{figure}[h!]
\centering
\begin{minipage}{0.49\textwidth}
\includegraphics[width=1.02\textwidth]{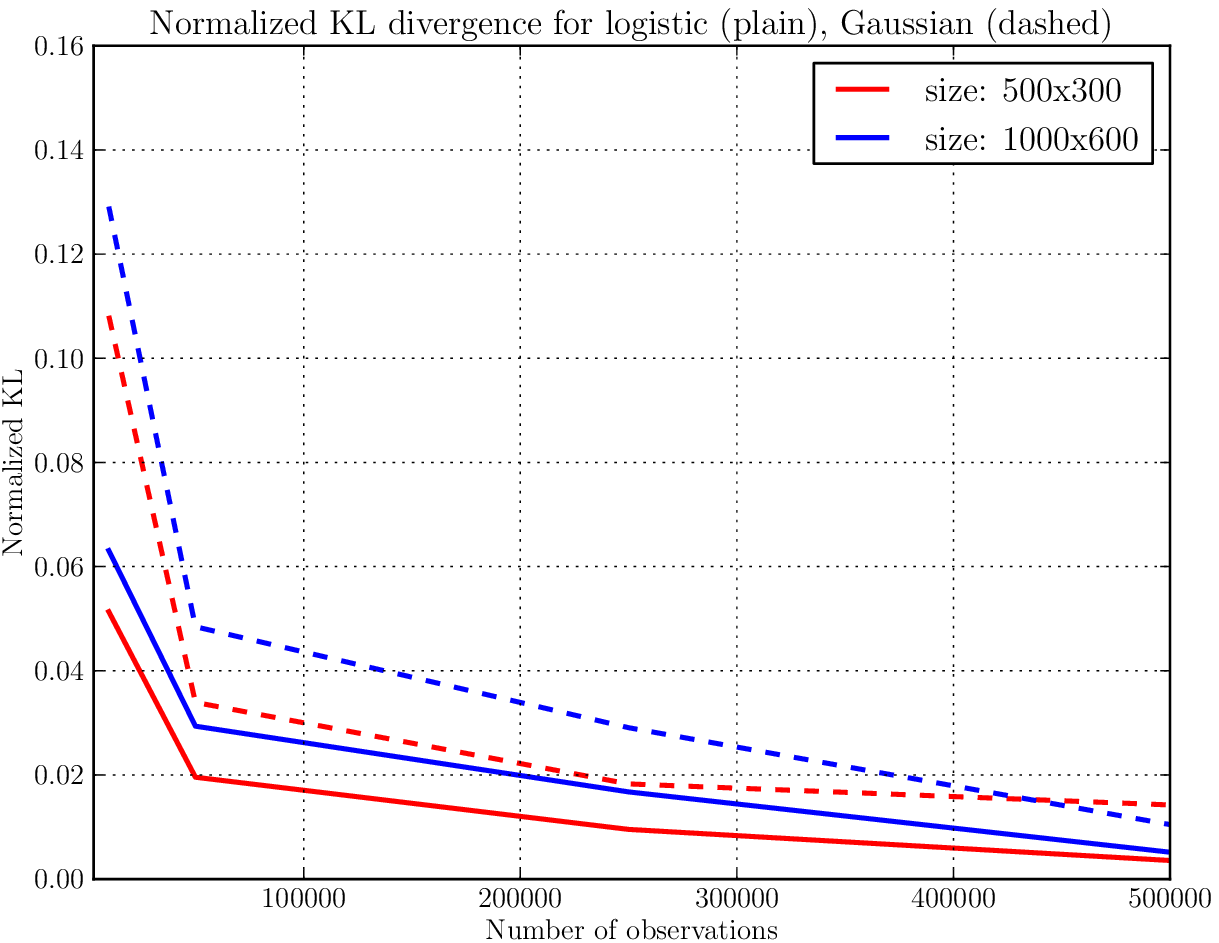}
\end{minipage}
\begin{minipage}{0.49\textwidth}
 \includegraphics[width=1.02\textwidth]{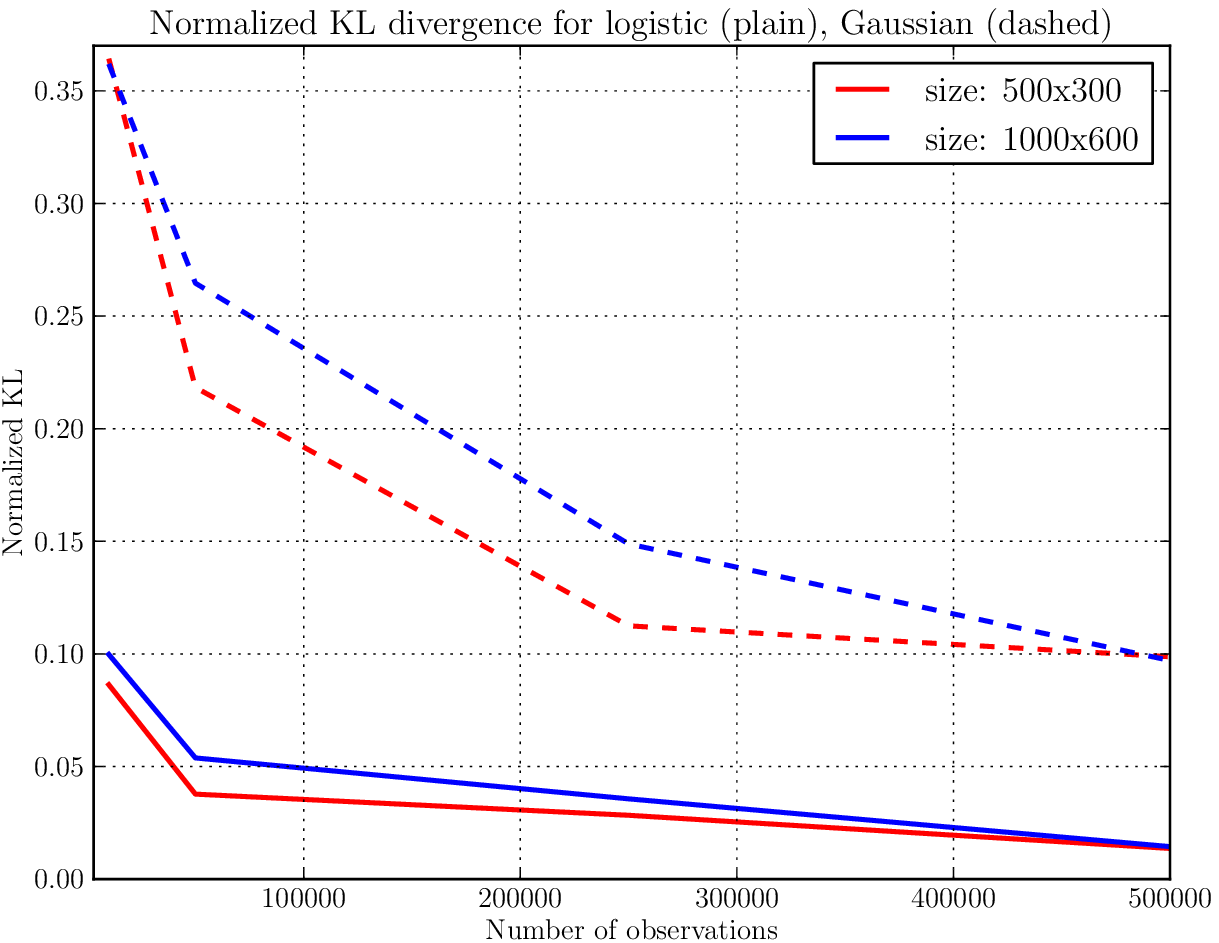}
\end{minipage}
\caption{ Kullback-Leibler divergence between the estimated and the true model
for different matrices sizes and sampling fraction, normalized by number of classes.
Right figure: binomial and the Gaussian models ;
left figure: multinomial with five classes and Gaussian model.}
\label{fig:KL}
\end{figure}

\begin{table}[h!]
\centering
\begin{tabular}{|l|c|c|c|c|}
   \hline
   \textbf{Number of observations} & $10\cdot 10^3$ &$50\cdot 10^3$ &$250\cdot 10^3$ &$500\cdot 10^3$   \\
   \hline
   \textbf{Gaussian prediction error} & $0.50$ &$0.38$ & $0.32$ & $0.32$  \\
   \hline
   \textbf{Logistic prediction error}& $0.46$ & $0.33$ & $0.31$ & $0.31$  \\
   \hline
\end{tabular}
\caption{Prediction errors for a binomial (2 classes) underlying model, for a $1000 \times 600$ matrix.}
\label{tab:sim_bin_error}
\end{table}

\begin{table}[h!]
\centering
\begin{tabular}{|l|c|c|c|c|}
   \hline
   \textbf{Number of observations} & $10\cdot 10^3$ &$50\cdot 10^3$ &$250\cdot 10^3$ &$500\cdot 10^3$   \\
   \hline
   \textbf{Gaussian prediction error} & $0.75$ &$0.75$ & $0.72$ & $0.71$  \\
   \hline
   \textbf{Logistic prediction error}& $0.75$ & $0.67$ & $0.58$ & $0.57$  \\
   \hline
\end{tabular}
\caption{Prediction Error for a multinomial (5 classes) distribution against a $1000 \times 600$ matrix.}
\label{tab:sim_mult_error}
\end{table}

We have also run the same estimators on the MovieLens $100k$ dataset. In this case,
the Kullback-Leibler divergence cannot be computed. Therefore, to assess the prediction errors,
we randomly select $20\%$ of the entries as a test set,
and the remaining entries are split between a training set ($80\%$) and a validation set ($20\%$).

For this dataset, ratings range from $1$ to $5$. To consider the benefit of a binomial model, we have tested
each rating against the others (\eg ratings $5$ are set to $0$ and all others are set to $1$).

These results are summarized in \Cref{tab:ml_bin}. For the multinomial case,
we find a prediction error of $0.59$ for the logistic model against a $0.63$ for the Gaussian one.

\begin{table*}[h]
\centering
\begin{tabular}{|l|c|c|c|c|c|}
   \hline
   \textbf{Rating against the others} & $1$ &$2$ &$3$ &$4$ & $5$   \\
   \hline
   \textbf{Gaussian prediction error} & $0.12$ &$0.20$ & $0.39$ & $0.46$ &  $0.30$\\
   \hline
   \textbf{Logistic prediction error}& $0.06$ &$0.11$ & $0.27$ & $0.34$ &$0.20$  \\
   \hline
\end{tabular}
\caption{Binomial prediction error when performing one versus the others procedure on the MovieLens $100k$
dataset.}
\label{tab:ml_bin}
\end{table*}

\section{Proofs of main results}
\label{sec:proof}
\subsection{Proof of \Cref{col:bin} and \Cref{th1}}
\label{proofth1}
\begin{proof}
Since \Cref{col:bin} is an application of \Cref{th1} for $\Class=2$ and $\qClass=1$ it suffices to prove \Cref{th1}.

We consider a tensor $\ten{X}$ which satisfies $\Obj(\ten{X}) \leq \Obj(\ten{\tX})$, (\eg $\ten{X}=\hat{\ten{X}}$).
We get from \Cref{InegBeg}
\begin{equation}
\label{th1:ineg1}
\Lik(\ten{X})-\Lik(\ten{\tX})\leq \lambda \sqrt{\rth } \sqrt{\KL{f(\ten{\tX})}{f(\ten{X})}} \eqs,
\end{equation}
where
\begin{equation}
\label{eq:definition-bar-r}
\bar{r} = \frac{2m_1m_2 \mrk{\ten{\tilde{X}}} }{K_\gamma} .
\end{equation}
Let us define
\begin{equation}\label{eq:def_D}
\D \left( f(\ten{\bar{X}}),f(\ten{X}) \right)\eqdef \EE\left[\left(\Lik(\ten{X})-\Lik(\ten{\bar{X}})\right) \right]\eqs,
\end{equation}
where the expectation is taken both over the $(E_i)_{1 \leq i \leq n}$ and $(Y_i)_{1 \leq i \leq n}$.
As stated in \Cref{D:K}, \Cref{A1} implies  $\mu \D\left(f(\ten{\tX}),f(\ten{X})\right) \geq   \KL{f(\ten{\tX})}{f(\ten{X})}$.
We now need to control the left hand side of \eqref{th1:ineg1} uniformly over $X$  with high probability.
Since we assume  $\lambda>2\max_{\qclass \in [\qClass]} \| \bar{\Sigma}^\qclass \|_{\sigma,\infty}$
applying \Cref{LemProj}~\eqref{NucKul}  and then \Cref{D:K} yields
\begin{equation}
\label{proof:th1ineq1}
\sum_{\qclass=1}^{\qClass}\|\mat{X}^\qclass-\mat{\bar{X}}^\qclass\|_{\sigma,1}
 \leq 4 \sqrt{\rth}\sqrt{\KL{f(\ten{\tX})}{f(\ten{X})}}
 \leq 4 \sqrt{\mu\rth} \sqrt{\D\left(f(\ten{\tX}),f(\ten{X})\right)} \eqs,
\end{equation}
Consequently, if we define $ \mathcal{C}(r)$ as
\begin{equation*}
 \mathcal{C}(r)\eqdef\left\{\ten{X} \in  \tenset{m_1}{m_2}{\qClass}
 : \:\: \sum_{\qclass=1}^{\qClass}\|\mat{X}^\qclass-\mat{\bar{X}}^\qclass\|_{\sigma,1}
 \leq \sqrt{ r \D\left(f(\ten{\tX}),f(\ten{X})\right) } \right\}\eqs,
\end{equation*}
we need to control $(\Lik(\ten{X})-\Lik(\ten{\tX}))$ for $\ten{X} \in \mathcal{C}(16 \mu \rth)$.
We have to ensure that $\D\left(f(\ten{\tX}),f(\ten{X})\right)$ is
greater than a given threshold $\beta >0$ and therefore we define the following set
\begin{equation}
\label{eq:cbeta}
 \mathcal{C}_\beta(r)=\left\{ \ten{X} \in \mathcal{C}(r), \:
 \D\left(f(\ten{\tX}),f(\ten{X})\right) \geq \beta \right\}\eqs.
\end{equation}
We then consider the two following cases.\\
\textbf{Case 1}. If $\D\left(f(\ten{\tX}),f(\ten{X})\right) > \beta$, \eqref{proof:th1ineq1} gives $X \in \mathcal{C}_\beta(16\mu\rth)$.
Plugging \Cref{DevUnifCont} in \eqref{th1:ineg1} with $\beta=2M_\gamma \sqrt{\log(d)}/(\eta\sqrt{n\log(\alpha)})$
, $\alpha=e$ and $\eta=1/(4\alpha)$ then it holds with probability at least
$1-2d^{-1}/(1-d^{-1})\geq 1-2/d$
\begin{equation*}
\frac{\D\left(f(\ten{\tX}),f(\ten{X})\right)}{2}-\epsilon(16\mu\rth,\alpha,\eta) \leq \lambda \sqrt{\rth} \sqrt{\KL{f(\ten{\tX})}{f(\ten{X})}}\eqs,
\end{equation*}
where $\epsilon$ is defined in \Cref{DevUnifCont}.
Recalling \Cref{D:K} we get
 \begin{equation*}
  \frac{\KL{f(\ten{\tX})}{f(\ten{X})}}{2\mu} -\lambda \sqrt{\rth } \sqrt{\KL{f(\ten{\tX}}{f(\ten{X})}} -\epsilon(16\mu\rth,\alpha,\eta) \leq 0\eqs.
 \end{equation*}
An analysis of this second order polynomial and the relation
$\epsilon(16\mu\rth,\alpha,\eta)/\mu=\epsilon(16\rth,\alpha,\eta)$ lead to

\begin{equation}
 \label{th1bound1}
 \sqrt{\KL{f(\ten{\tX})}{f(\ten{X})}} \leq \mu  \left( \lambda \sqrt{\rth } + \sqrt{\lambda^2\rth + 2 \epsilon(16\rth,\alpha,\eta)}\right) \eqs.
\end{equation}
Applying the inequality $(a+b)^2 \leq 2(a^2+b^2)$ gives the bound of \Cref{th1}.\\
\textbf{Case 2}. If $\D\left(f(\ten{\tX}),f(\ten{X})\right)\leq \beta$ then \Cref{D:K} yields
\begin{equation}
 \label{th1bound2}
 \KL{f(\ten{\tX})}{f(\ten{X})}\leq \mu \beta \eqs.
\end{equation}
Combining \eqref{th1bound1} and \eqref{th1bound2} concludes the proof.
\end{proof}


For $\mat{X} \in  \matset{m_1}{m_2}$,  denote by $\mathcal{S}_1(\mat{X}) \subset \RR^{m_1}$ (\resp $\mathcal{S}_2(\mat{X}) \subset \RR^{m_2}$)
the linear spans generated by left (\resp right) singular vectors of $\mat{X}$.
$P_{\mathcal{S}^\bot_1(\mat{X})}$  (\resp $P_{\mathcal{S}^\bot_2(\mat{X})}$)
denotes the orthogonal projections on $\mathcal{S}^\bot_1(\mat{X})$ (\resp $\mathcal{S}^\bot_2(\mat{X})$).
We then define the following orthogonal projections on $\RR^{m_1 \times m_2}$
\begin{equation*}
\Proj_{\mat{X}}^\bot:\mat{\tilde{X}} \mapsto P_{\mathcal{S}^\bot_1(\mat{X})}\mat{\tilde{X}} P_{\mathcal{S}^\bot_2(\mat{X})}
\text{ and } \Proj_{\mat{X}}: \mat{\tilde{X}} \mapsto \mat{\tilde{X}}-\Proj_{\mat{X}}^\bot(\mat{\tilde{X}})\eqs.
\end{equation*}

\begin{lemma}
 \label{InegBeg}
 Let $\ten{X},\ten{\tilde{X}} \in \tenset{m_1}{m_2}{\qClass}$ satisfying  $\Obj(\ten{X}) \leq \Obj(\ten{\tilde{X}})$, then
 \begin{equation*}
\Lik(\ten{X})-\Lik(\ten{\tilde{X}})\leq \lambda \bar{r}^{1/2}  \sqrt{\KL{f(\ten{\tilde{X}})}{f(\ten{X})}} \eqs,
 \end{equation*}
 where $\bar{r}$ is defined in \eqref{eq:definition-bar-r}.
\end{lemma}

\begin{proof}
Since $\Obj(\ten{X}) \leq \Obj(\ten{\tilde{X}})$, we obtain
 \begin{align*}
\Lik(\ten{X})-\Lik(\ten{\tilde{X}}) \leq &\lambda \sum_{\qclass=1}^{\qClass} (\|\mat{\tilde{X}}^\qclass \|_{\sigma,1}-\|\mat{X}^\qclass\|_{\sigma,1})
\leq \lambda \sum_{\qclass=1}^{\qClass} \|\Proj_{\mat{\tilde{X}}^\qclass}(\mat{X}-\mat{\tilde{X}}^\qclass)\|_{\sigma,1} \eqs, \\
					    \leq &\lambda \sqrt{2 \mrk{\ten{\tilde{X}}} } \left( \sum_{\qclass=1}^{\qClass} \|\mat{X}-\mat{\tilde{X}}\|_{\sigma,2}\right)\eqs,\\
\end{align*}
where we have used \Cref{algebricres}-\eqref{ProjRel} and \eqref{diffschatten} and  for the last two lines and
the definition of $K_\gamma$ and \Cref{hel:kul} to get the result.
\end{proof}

\begin{lemma}
\label{algebricres}
 For any pair of matrices $\mat{X}, \, \mat{\tilde{X}} \in \matset{m_1}{m_2}$ we have
 \begin{enumerate}[(i)]
  \item $ \|\mat{X} + \Proj_{\mat{X}}^\bot(\mat{\tilde{X}}) \|_{\sigma,1}=\|\mat{X}\|_{\sigma,1} +  \|\Proj_{\mat{X}}^\bot(\mat{\tilde{X}}) \|_{\sigma,1}\eqs,$ \label{PenRel}
  \item $\| \Proj_{\mat{X}}(\mat{\tilde{X}}) \|_{\sigma,1} \leq \sqrt{2 \rank(\mat{X})} \|\mat{\tilde{X}}\|_{\sigma,2}\eqs,$ \label{ProjRel}
  \item  $ \|\mat{X}\|_{\sigma,1}-\|\mat{\tilde{X}}\|_{\sigma,1} \leq  \|\Proj_{\mat{X}}(\mat{\tilde{X}}-\mat{X})\|_{\sigma,1}\eqs.$ \label{diffschatten}
  \end{enumerate}

\end{lemma}
 \begin{proof}
If $A, B \in \rset^{m_1 \times m_2}$ are two matrices satisfying $\mathcal{S}_i(A) \perp \mathcal{S}_i(B)$, $i=1,2$,
then $\| A + B \|_{\sigma,1}= \| A \|_{\sigma,1} + \| B \|_{\sigma,1}$. Applying this identity with $A = X$ and $B= \Proj_{\mat{X}}^\bot(\mat{\tilde{X}})$, we obtain
 \[
 \|X + \Proj_{\mat{X}}^\bot(\mat{\tilde{X}}) \|_{\sigma,1}
 = \|X\|_{\sigma,1} +  \|\Proj_{\mat{X}}^\bot(\mat{\tilde{X}}) \|_{\sigma,1}\eqs,
 \]
 showing~\eqref{PenRel}.

 It follows from the definition that
 $\Proj_{\mat{X}}(\mat{\tilde{X}})=P_{\mathcal{S}_1(\mat{X})}\mat{\tilde{X}} P_{\mathcal{S}^\bot_2(\mat{X})}+\mat{\tilde{X}}P_{\mathcal{S}_2(\mat{X})}$.
 Note that $\Proj_{\mat{X}}$ is an orthogonal projector on $\rset^{m_1\times m_2}$ equipped with the euclidean product $\pscal{\cdot}{\cdot}$.
 On the other hand, the Cauchy-Schwarz inequality implies that for any matrix $C$,  $\|C\|_{\sigma,1} \leq \sqrt{ \rank(C)} \|C\|_{\sigma,2}$. Consequently  \eqref{ProjRel} follows from
 \begin{align*}
 \| \Proj_{\mat{X}}(\mat{\tilde{X}}) \|_{\sigma,1} &\leq \sqrt{2 \rank(X)} \|\Proj_{\mat{X}}(\mat{\tilde{X}})\|_{\sigma,2}\eqs \leq \sqrt{2 \rank(X)} \|\mat{\tilde{X}}\|_{\sigma,2}\eqs \eqsp.
 \end{align*}
 Finally, since $\mat{\tilde{X}}=\mat{X} + \Proj_{\mat{X}}^\bot(\mat{\tilde{X}}-\mat{X})+\Proj_{\mat{X}}(\mat{\tilde{X}}-\mat{X})$ we have
 \begin{align*}
  \|\mat{\tilde{X}} \|_{\sigma,1} &\geq \|\mat{X} + \Proj_{\mat{X}}^\bot(\mat{\tilde{X}}-\mat{X}) \|_{\sigma,1} -\|\Proj_{\mat{X}}(\mat{\tilde{X}}-\mat{X}) \|_{\sigma,1} \eqs, \\
  &= \|\mat{X} \|_{\sigma,1} + \| \Proj_{\mat{X}}^\bot(\mat{\tilde{X}}-\mat{X}) \|_{\sigma,1} -\|\Proj_{\mat{X}}(\mat{\tilde{X}}-\mat{X}) \|_{\sigma,1} \eqs,
 \end{align*}
 leading to \eqref{diffschatten}.
 \end{proof}

\begin{lemma}
\label{hel:kul}
For any tensor $\ten{X},\ten{\tilde{X}} \in \tenset{m_1}{m_2}{\qClass}$  and
$\Class$-link function $f$ it holds:
\begin{equation*}
  \dhes{f(\ten{X} ) }{f(\ten{\tilde{X}} )} \leq \KL{f(\ten{X})}{f(\ten{\tilde{X}})}
\end{equation*}
\end{lemma}

\begin{proof}
See \cite[Lemma 4.2]{Tsybakov09}
\end{proof}

\begin{lemma}
\label{HelFro}
For any $\Class, \qClass>0$ and $\Class$-link function $f$
and any $\ten{X},\ten{\tilde{X}}  \in  \tenset{m_1}{m_2}{\qClass}$
satisfying $\|\ten{X}\|_{\infty}\leq \gamma$ and $\|\ten{\tilde{X}}\|_{\infty}\leq \gamma$, we get:
 \begin{equation*}
 \sum_{\qclass=1}^{\qClass} \|\mat{X^\qclass}-\mat{\tilde{X^\qclass}}\|^2_{\sigma,2} \leq \frac{m_1 m_2}{K_\gamma}
 \dhes{f(\ten{X})}{f(\ten{\tilde{X}}}\leq \frac{m_1 m_2}{K_\gamma} \KL{f(\ten{X})}{f(\ten{\tilde{X}}} \eqs.
 \end{equation*}
\end{lemma}
\begin{proof}
For $\Class=2$  and $\qClass=1$, it is a consequence of \Cref{rem:K_gamma} and  \Cref{hel:kul}.
Otherwise, the proof follows from the definition  \eqref{eq:definition:K-gamma} 
of $K_\gamma$ and \Cref{hel:kul}.
\end{proof}

\begin{lemma}
\label{LemProj}
Let $\ten{X},\ten{\tilde{X}} \in \tenset{m_1}{m_2}{\qClass}$ satisfying $\|\ten{X}\|_{\infty}\leq \gamma$ and $\|\ten{\tilde{X}}\|_{\infty}\leq \gamma$.
Assume that $\lambda>2\max_{\qclass \in [\qClass]} \| \glik^\qclass(\mat{\tilde{X}}) \|_{\sigma,\infty}$
and $\Obj(\mat{X}) \leq \Obj(\mat{\tilde{X}})$. Then
\begin{align}
&\sum_{\qclass=1}^{\qClass}\|\Proj_{\mat{\tilde{X}}^\qclass}^\bot(\mat{X}^\qclass-\mat{\tilde{X}}^\qclass)\|_{\sigma,1} \leq 3 \sum_{\qclass=1}^{\qClass} \|\Proj_{\mat{\tilde{X}}^\qclass}(\mat{X}^\qclass-\mat{\tilde{X}}^\qclass)\|_{\sigma,1}\eqs,  \label{projrel} \\
&\sum_{\qclass=1}^{\qClass}\|\mat{X}^\qclass-\mat{\tilde{X}}^\qclass\|_{\sigma,1}
\leq 4 \sqrt{2 \rank(\ten{\tilde{X}})}\sum_{\qclass=1}^{\qClass} \|(\mat{X}^\qclass-\mat{\tilde{X}}^\qclass)\|_{\sigma,2}\eqs, \label{NucFob} \\
&\sum_{\qclass=1}^{\qClass}\|\mat{X}^\qclass-\mat{\tilde{X}}^\qclass\|_{\sigma,1}
\leq 4 \sqrt{ 2 m_1 m_2  \rank(\ten{\tilde{X}})/ K_\gamma}   \dhel{f(\ten{\tilde{X}}),f(\ten{X})}\eqs, \label{NucHel} \\
&\sum_{\qclass=1}^{\qClass}\|\mat{X}^\qclass-\mat{\tilde{X}}^\qclass\|_{\sigma,1}
 \leq 4 \sqrt{ 2 m_1 m_2  \rank(\ten{\tilde{X}})/ K_\gamma }  \sqrt{\KL{f(\ten{\tilde{X}})}{f(\ten{X})}}\eqs. \label{NucKul}
\end{align}
\end{lemma}

\begin{proof}
Since $\Obj(\ten{X}) \leq \Obj(\ten{\tilde{X}})$, we have
 \begin{equation*}
 \Lik(\ten{\tilde{X}}) -\Lik(\ten{X})  \geq \lambda \sum_{\qclass=1}^{\qClass} (\|\mat{X}^\qclass\|_{\sigma,1}-\|\mat{\tilde{X}}^\qclass\|_{\sigma,1}).
 \end{equation*}
For any $\mat{X} \in \rset^{m_1 \times m_2}$, using $\mat{X}=\mat{\tilde{X}}+\Proj_{\mat{\tilde{X}}}^\bot(\mat{X}-\mat{\tilde{X}})+\Proj_{\mat{\tilde{X}}}(\mat{X}-\mat{\tilde{X}})$, \Cref{algebricres}-\eqref{PenRel} and the triangular inequality, we get
\begin{equation*}
 \|\mat{X}\|_{\sigma,1} \geq \|\mat{\tilde{X}}\|_{\sigma,1} + \|\Proj_{\mat{\tilde{X}}}^\bot(\mat{X}-\mat{\tilde{X}})\|_{\sigma,1}
 -\|\Proj_{\mat{\tilde{X}}}(\mat{X}-\mat{\tilde{X}})\|_{\sigma,1} \eqs,
\end{equation*}
which implies
\begin{equation}
 \label{LemProj:lb}
 \Lik(\ten{\tilde{X}}) -\Lik(\ten{X})  \geq  \lambda
  \sum_{\qclass=1}^{\qClass} \left( \|\Proj_{\mat{\tilde{X}}^\qclass}^\bot(\mat{X}^\qclass-\mat{\tilde{X}}^\qclass)\|_{\sigma,1}
  -\|\Proj_{\mat{\tilde{X}}^\qclass}(\mat{X}-\mat{\tilde{X}}^\qclass)\|_{\sigma,1} \right)\eqs.
\end{equation}
Furthermore by concavity of $\Lik$ we have
\begin{align*}
 \Lik(\ten{\tilde{X}}) -\Lik(\ten{X})  \leq \sum_{\qclass=1}^{\qClass} \pscal{\glik^{\qclass}(\ten{\tilde{X}})}{\mat{\tilde{X}}^\qclass-\mat{X}^\qclass} \eqs.
\end{align*}
The duality between $\|\cdot\|_{\sigma,1}$ and $\|\cdot\|_{\sigma,\infty}$ (see for instance
\cite[Corollary IV.2.6]{Bhatia97}) leads to
\begin{align}
 \Lik(\ten{\tilde{X}}) -\Lik(\ten{X}) &\leq \max_{\qclass \in [\qClass]} \| \glik^\qclass(\mat{\tilde{X}}) \|_{\sigma,\infty}
 \sum_{\qclass=1}^{\qClass} \|\mat{\tilde{X}}^\qclass-\mat{X}^\qclass \|_{\sigma,1}\eqs,\nonumber \\
& \leq \frac{\lambda}{2} \sum_{\qclass=1}^{\qClass} \|\mat{\tilde{X}}^\qclass-\mat{X}^\qclass \|_{\sigma,1} \eqs, \nonumber \\
& \leq \frac{\lambda}{2}  \sum_{\qclass=1}^{\qClass}
(\|\Proj_{\mat{\tilde{X}}^\qclass}^\bot(\mat{X}^\qclass-\mat{\tilde{X}}^\qclass)\|_{\sigma,1}
+ \|\Proj_{\mat{\tilde{X}}^\qclass}(\mat{X}^\qclass-\mat{\tilde{X}}^\qclass)\|_{\sigma,1})\eqs, \label{LemProj:ub}
\end{align}
where we used $\lambda>2\max_{\qclass \in [\qClass]} \| \glik^\qclass(\mat{\tilde{X}}) \|_{\sigma,\infty}$
in the second line. Then combining \eqref{LemProj:lb} with \eqref{LemProj:ub} gives \eqref{projrel}.
Since for any $\qclass \in [\qClass]$,
$\mat{X}^\qclass-\mat{\tilde{X}}^\qclass=\Proj_{\mat{\tilde{X}}^\qclass}^\bot(\mat{X}^\qclass-\mat{\tilde{X}}^\qclass)
+\Proj_{\mat{\tilde{X}}^\qclass}(\mat{X}^\qclass-\mat{\tilde{X}}^\qclass)$,
using the triangular inequality and \eqref{projrel} yields
\begin{equation}
\label{LemProj:int}
  \sum_{\qclass=1}^{\qClass} \|\mat{X}^\qclass-\mat{\tilde{X}}^\qclass\|_{\sigma,1} \leq
 4 \|\Proj_{\mat{\tilde{X}}^\qclass}(\mat{X}^\qclass-\mat{\tilde{X}}^\qclass)\|_{\sigma,1}.
\end{equation}
Combining \eqref{LemProj:int} and  \eqref{projrel} immediately leads to \eqref{NucFob} and \eqref{NucHel} is a consequence of  \eqref{NucFob} and
the definition of $K_\gamma$. The
statement \eqref{NucKul} follows from \eqref{NucHel} and \Cref{hel:kul}.
\end{proof}

\begin{lemma}
\label{D:K}
 Under \Cref{A1} we have
 \begin{equation*}
 \D\left(f(\ten{\tX}),f(\ten{X})\right) \geq  \frac{1}{\mu} \KL{f(\ten{\tX})}{f(\ten{X})} \eqs.
\end{equation*}
where $\D(\cdot,\cdot)$ is defined in \eqref{eq:def_D}.
\end{lemma}
\begin{proof}
Follows from
\begin{align*}
\D\left(f(\ten{\tX}),f(\ten{X})\right) &=\frac{1}{n} \sum_{i=1}^{n}\sum_{\substack{k \in [m_1] \\ l \in [m_2]}} \sum_{\class \in [\Class]} \pi_{k,l}
\left[f^\class(\ten{\tX}_{k,l})\log \left(\frac{f^\class(\ten{\tX}_{k,l})}{f^\class(\ten{X}_{k,l}) } \right) \right] \eqs,\\
&\geq \frac{1}{\mu m_1 m_2} \sum_{\substack{k \in [m_1] \\ l \in [m_2]}}  \sum_{\class \in [\Class]}
\left[f^\class(\ten{\tX}_{k,l})\log \left(\frac{f^\class(\ten{\tX}_{k,l})}{f^\class(\ten{X}_{k,l}) } \right) \right]\eqs.
\end{align*}
\end{proof}

 \begin{lemma}
 \label{DevUnifCont}
 Assume that $\lambda \geq \bar{\Sigma}$.  Let $\alpha>1$, $\beta>0$ and $0<\eta<1/2\alpha$.
Then with probability at least
  $$
  1-2(\exp(-n\eta^2\log(\alpha)\beta^2/(4M_\gamma^2))/(1-\exp(-n\eta^2\log(\alpha)\beta^2/(4M_\gamma^2)))
  $$
  we have for all $\ten{X} \in \mathcal{C}_\beta (r)$:
 \begin{equation*}
   |\Lik(\ten{X})-\Lik(\ten{\tX}) - \D\left(f(\ten{\tX}),f(\ten{X})\right)|\leq \frac{\D\left(f(\ten{\tX}),f(\ten{X})\right)}{2}+\epsilon(r,\alpha,\eta)\eqs ,
 \end{equation*}
where
\begin{equation}\label{eq:def_eps}
 \epsilon(r,\alpha,\eta) \eqdef\frac{4 \qClass^2 L_\gamma^2r}{1/(2\alpha)-\eta}(\EE \|\Sigma_R \|_{\sigma,\infty})^2\eqs,
\end{equation}
 and $\mathcal{C}_\beta (r)$ is defined in \eqref{eq:cbeta}.
\end{lemma}

\begin{proof}
The proof is adapted from \cite[Theorem~1]{Negahban_Wainwright12} and \cite[Lemma~12]{Klopp14}.
We use a peeling argument combined with a sharp deviation inequality detailed in \Cref{SliceCont},
Consider the events
\begin{multline*}
 \mathcal{B} \eqdef \bigg\{ \exists \ten{X} \in \mathcal{C}_\beta (r)\bigg |\\
  |\Lik(\ten{X})-\Lik(\ten{\tX})- \D\left(f(\ten{\tX}),f(\ten{X})\right) |
 > \frac{\D\left(f(\ten{\tX}),f(\ten{X})\right)}{2}+\epsilon(r,\alpha,\eta) \bigg\}\eqs,
\end{multline*}
and
\begin{equation*}
 \mathcal{S}_l\eqdef \left\{ \ten{X} \in \mathcal{C}_\beta (r)| \alpha^{l-1}\beta < \D\left(f(\ten{\tX}),f(\ten{X})\right) < \alpha^{l}\beta \right\}\eqs.
\end{equation*}
Let us also define the set
\begin{equation*}
  \mathcal{C}_\beta (r,t)=\left\{ \ten{X} \in  \RR^{m_1 \times m_2}|\:\: \ten{X}  \in \mathcal{C}_\beta (r),\:  \D\left(f(\ten{\tX}),f(\ten{X})\right) \leq t \right\} \eqs,
\end{equation*}
and
\begin{equation}\label{eq:def_Zt}
 Z_t \eqdef \sup_{\ten{X} \in \mathcal{C}_\beta (r,t)} |\Lik(\ten{X})-\Lik(\ten{\tX}) - \D\left(f(\ten{\tX}),f(\ten{X})\right) | \eqs,
\end{equation}
Then for any $\ten{X} \in  \mathcal{B} \cap \mathcal{S}_l$ we have
\begin{equation*}
 |\Lik(\ten{X})-\Lik(\ten{\tX}) - \D\left(f(\ten{\tX}),f(\ten{X})\right) | > \frac{1}{2}\alpha^{l-1}\beta + \epsilon(r,\alpha,\eta)\eqs,
\end{equation*}
Moreover by definition of $\mathcal{S}_l$, $\ten{X} \in \mathcal{C}_\beta (r,\alpha^{l}\beta)$.
Therefore

\begin{equation*}
 \mathcal{B} \cap \mathcal{S}_l \subset \mathcal{B}_l \eqdef \{ Z_{\alpha^{l}\beta} > \frac{1}{2\alpha}\alpha^{l}\beta+\epsilon(r,\alpha,\eta) \}\eqs,
\end{equation*}
If we now apply the union bound and \Cref{SliceCont} we get
 \begin{equation*}
 \PP(\mathcal{B}) \leq \sum_{l=1}^{+\infty} \PP(\mathcal{B}_l)\leq \sum_{l=1}^{+\infty} \exp\left(-\frac{n\eta^2(\alpha^{l}\beta)^2}{8M_\gamma^2}\right)\leq \frac{\exp(-\frac{n\eta^2\log(\alpha)\beta^2}{4M_\gamma^2})}{1-\exp(-\frac{n\eta^2\log(\alpha)\beta^2}{4M_\gamma^2})}\eqs,
 \end{equation*}
 where we used $x\leq \rme^x$ in the second inequality.
\end{proof}

\begin{lemma}
 \label{SliceCont}
Assume that $\lambda \geq \bar{\Sigma}$. Let $\alpha>1$ and $0<\eta<\frac{1}{2\alpha}$. Then we have
\begin{equation}
\label{eq:SliceCont}
 \PP\left(Z_t > t/(2\alpha) + \epsilon(r,\alpha,\beta) \right) \leq \exp\left(-n\eta^2t^2/(8M_\gamma^2)\right)\eqs,
\end{equation}
where $Z_t$ and $\epsilon(r,\alpha,\eta)$ are defined in \eqref{eq:def_Zt} and \eqref{eq:def_eps}, respectively.
 \end{lemma}

 \begin{proof}
Using Massart's inequality (\cite[Theorem 9]{Massart00})
we get for  $0<\eta<1/(2\alpha)$
\begin{equation}
\label{proof:MassartConc}
 \PP(Z_t> \EE[Z_t]+\eta t)\leq \exp\left(-\eta^2nt^2/(8M_\gamma^2) \right)\eqs.
\end{equation}
By using the standard symmetrization argument, we get
\begin{equation*}
 \EE[Z_t]\leq 2 \EE\left[\sup_{\ten{X} \in \mathcal{C}_\beta (r,t)} \left|\frac{1}{n}\sum_{i=1}^{n}
 \varepsilon_i \sum_{\class=1}^{\Class} \1_{\{Y_i=j\}} \log\left(\frac{f^\class(\ten{X}_i)}{f^\class(\ten{\tX}_i)}\right)\right|\right]\eqs,
\end{equation*}
where $\bvarepsilon \eqdef (\varepsilon_i)_{1 \leq i\leq n}$ is a Rademacher sequence which is independent from $(Y_i)_{1 \leq i\leq n}$ and $(E_i)_{1 \leq i\leq n}$. \Cref{A3} yields
\begin{equation*}
 \EE[Z_t]\leq \sum_{\qclass=1}^\qClass 2 \EE\left[\sup_{\ten{X} \in \mathcal{C}_\beta (r,t)} \left|\frac{1}{n}\sum_{i=1}^{n}
 \varepsilon_i \sum_{\class=1}^{\Class} \1_{\{Y_i=j\}} \log\left(\frac{g_\qclass^\class(\mat{X}^\qclass_i)}{g_\qclass^\class(\mat{\tX}^\qclass_i)}\right)\right|\right]\eqs.
\end{equation*}
Since for any $i \in [n]$, the function
\begin{equation*}
 \phi_i(x) \eqdef \frac{1}{L_\gamma} \sum_{\class=1}^{\Class} \1_{\{Y_i=j\}} \log\left(\frac{g_\qclass^\class(x+\mat{\tX}^\qclass_i)}{g_\qclass^\class(\mat{\tX}^\qclass_i)}\right)
\end{equation*}
is a contraction satisfying $\phi_i(0)=0$, the contraction principle (\cite[Theorem 4.12]{Ledoux_Talagrand91})
and the fact that $(\varepsilon_i)_{i=1}^n$ is independent from $(Y_i)_{i=1}^n$ and $(\omega_i)_{i=1}^n$ yields
\begin{equation*}
 \EE[Z_t] 
\leq 4 L_\gamma \sum_{\qclass=1}^\qClass \EE\left[\sup_{\ten{X} \in \mathcal{C}_\beta (r,t)} \left|\frac{1}{n}\sum_{i=1}^{n} \varepsilon_i\pscal{ \mat{X}^\qclass- \mat{\tX}^\qclass}{E_i} \right|\right]=
\end{equation*}
Denoting $\Sigma_R \eqdef n^{-1} \sum_{i=1}^{n} \varepsilon_i E_i$ and the duality, the previous inequality implies
\begin{align*}
 \EE[Z_t] &\leq 4 L_\gamma  \sum_{\qclass=1}^\qClass \EE \left[\sup_{\ten{X} \in \mathcal{C}_\beta (r,t)} \left| \pscal{\mat{X}^\qclass- \mat{\tX}^\qclass}{\Sigma_R}  \right|\right]\\
 &\leq 4 L_\gamma  \sum_{\qclass=1}^\qClass \EE\left[\sup_{\ten{X} \in \mathcal{C}_\beta (r,t)}\|\mat{X}^\qclass- \mat{\tX}^\qclass \|_{\sigma,1}\|\Sigma_R \|_{\sigma,\infty} \right]
 \leq 4\qClass L_\gamma  \EE [\|\Sigma_R \|_{\sigma,\infty}] \sqrt{rt}\eqs,
\end{align*}
where we have  the definition of $ \mathcal{C}_\beta (r,t)$
for the last inequality. Plugging into \eqref{proof:MassartConc} gives
\begin{equation*}
 \PP(Z_t> 4 \qClass L_\gamma \EE [\|\Sigma_R \|_{\sigma,\infty}] \sqrt{rt}+\eta t)\leq \exp\left(-\eta^2nt^2/(8M_\gamma^2)\right)\eqs.
\end{equation*}
The proof is concluded by noting that, since for any $a,b \in \RR  $ and $c>0$, $ab \leq (a^2/c +cb^2)/2$,
\begin{equation*}
4 \qClass L_\gamma \EE [\|\Sigma_R \|_{\sigma,\infty}] \sqrt{rt}\leq \frac{1}{1/(2\alpha)-\eta}4\qClass^2 L_\gamma^2r\EE [\|\Sigma_R \|_{\sigma,\infty}]^2+(1/(2\alpha)-\eta)t\eqs.
\end{equation*}
 \end{proof}

\subsection{Proof of \Cref{th2}}
\label{proofth2}
\begin{proof}
By \Cref{th2} it suffices to control $ \| \bar{\Sigma}^\qclass \|_{\sigma,\infty}$
and $\EE [\norm{\Sigma_R}_{\sigma,\infty}]$. For any $\qclass \in [\qClass]$, by definition
 \begin{equation*}
  \bar{\Sigma}^\qclass =-\frac{1}{n}
  \sum_{i=1}^{n}\left( \sum_{\class=1}^{\Class}
  \1_{\{Y_i=\class\}}
  \frac{\partial_\qclass f^\class(\ten{\tX}_i)}{f^\class(\ten{\tX}_i)}
  \right) E_i \eqs,
 \end{equation*}
 with $\partial_\qclass$ designating the partial derivative against the $l$-th variable.
 The sequence of matrices
 \begin{equation*}
  Z_i\eqdef \left( \sum_{\class=1}^{\Class}
  \1_{\{Y_i=\class\}}
  \frac{\partial_\qclass f^\class(\ten{\tX}_i)}{f^\class(\ten{\tX}_i)}
  \right) E_i
  =\left( \sum_{\class=1}^{\Class}
  \1_{\{Y_i=\class\}}
  \frac{(g^\class_\qclass)' (\mat{\tX}^\qclass_i)}{g^\class_\qclass (\mat{\tX}^\qclass_i)}
  \right) E_i
 \end{equation*}
satisfies $\EE[Z_i]=0$ (as any score function) and $\norm{\mat{Z_{i}}}_{\sigma,\infty} \leq L_\gamma$.\\
Noticing $e_k(e'_{k'})^\top (e_k(e'_{k'})^\top)^\top=e_k(e'_{k'})^\top$  we also get
  \begin{equation*}
\frac{1}{n} \sum_{i=1}^{n} \EE[ \mat{Z_{i}}\mat{Z_{i}}^\top ]=
    \sum_{k=1}^{m_1}\left( \sum_{k'=1}^{m_2}\pi_{k,k'} \left(
 \sum_{\class=1}^{\Class}
  f^\class(\ten{\tX}_{k,k'})
  \left(\frac{\partial_\qclass f^\class(\ten{\tX}_{k,k'})}{f^\class(\ten{\tX}_{k,k'})}\right)^2
   \right)\right)e_k(e'_{k})^\top \eqs,
  \end{equation*}
which is diagonal.
We recall the definition
$C_{k'}\!=\sum_{k=1}^{m_1} \pi_{k,k'}$ and $R_k\!=\sum_{k'=1}^{m_2} \pi_{k,k'}$ for any $k' \in [m_2]$, $k\in [m_1]$.
Since
\begin{equation*}
\left(\frac{\partial_\qclass f^\class(\ten{\tX}_{k,k'})}{f^\class(\ten{\tX}_{k,k'})}\right)^2
\leq L^2_\gamma \eqs,
\end{equation*}
and $(f^\class(\ten{\tX}_{k,k'}))_{\class=1}^\Class$ is a probability distribution,
 we obtain
\begin{equation*}
\norm{\EE\left[\frac{1}{n} \sum_{i=1}^{n}  \mat{Z_{i}}\mat{Z_{i}}^\top \right]}_{\sigma,\infty}
\leq  L_\gamma^2 \norm{\diag((R_k)_{k=1}^{m_1})}_{\sigma,\infty} \leq  L_\gamma^2 \frac{\Lc}{m} \eqs,
\end{equation*}
were we have \Cref{A2} for the last inequality.
Using a similar argument
we get $\|\EE[\sum_{i=1}^{n} \mat{Z_{i}}^\top\mat{Z_{i}} ]\|_{\sigma,\infty}/n\leq L_\gamma^2 \Lc/m$.
Therefore, \Cref{prop:bernstein} applied with $t=\log(d)$, $U=L_\gamma$ and $\sigma_Z^2 =L^2_\gamma\Lc/m$
yields with at least probability $1-1/d$,
 \begin{equation}
 \label{proof:th2ineq1}
  \norm{ \glik^\qclass(\mat{\tilde{X}})}_{\sigma,\infty}
  \leq (1+\sqrt{3}) L_\gamma\max \left\{ \sqrt{ \frac{2\Lc\log(d)}{mn}} , \frac{2}{3} \frac{\log(d)}{n} \right \} \eqs.
  \end{equation}
With the same analysis for $\Sigma_R\eqdef\frac{1}{n}\sum_{i=1}^{n} \varepsilon_i E_i$ and by applying \Cref{lem:MatExp}
with $U=1$ and $\sigma_Z^2 =\frac{\nu}{m}$, for $n\geq n^* \eqdef m\log(d)/(9 \Lc)$ it holds:
\begin{equation}
 \label{proof:th2ineq2}
  \EE \left[\|\Sigma_R  \|_{\sigma,\infty} \right] \leq c^{*}\sqrt{\frac{2e\Lc \log(d)}{mn}} \eqs.
\end{equation}
Assuming $n \geq 2 m \log(d)/(9\nu)$, implies $n \geq n^*$ and \eqref{proof:th2ineq2} is therefore satisfied.
Since it also implies $ \sqrt{2\nu\log(d)/(mn)} \geq 2\log(d)/(3n)$, the second term of \eqref{proof:th2ineq1} is negligible.
Consequently taking $\lambda\geq2 (1+\sqrt{3}) L_\gamma\sqrt{2\nu\log(d)/(mn)}$,
a union bound argument ensures that
$\lambda>2\max_{\qclass \in [\qClass]} \| \glik^\qclass(\mat{\tilde{X}}) \|_{\sigma,\infty}$
with probability at least $1-\qClass/d$.\\
By taking $\lambda$, $\beta$ and $n$  as in \Cref{th2} statement , with probability larger than $1-(2+\qClass)/d$,
 \Cref{th1} result holds when replacing  $\EE \|\Sigma_R  \|_{\sigma,\infty} $ by its upper bound \eqref{proof:th2ineq2}.
 Using the inequality $(a+b)^2 \leq 2(a^2 +b^2)$ yields the result with $\bar{c}=24832$.
\end{proof}

\begin{proposition}
\label{prop:bernstein}
 Consider a finite sequence of independent random matrices $(\Bern{Z_{i}})_{1 \leq i \leq n}\in \RR^{m_1 \times m_2}$ satisfying $\EE[\Bern{Z_{i}}]=0$ and for some $U>0$,
 $\| \Bern{Z_{i}} \|_{\sigma,\infty} \leq U$ for all $i= 1, \dots, n$. Then for any $t>0$
 \begin{equation*}
 \PP\left( \left\|\frac{1}{n} \sum_{i=1}^{n} \Bern{Z_{i}}  \right\|_{\sigma,\infty} > t \right) \leq d\exp \left(- \frac{nt^2/2}{\sigma^2_Z+Ut/3} \right) \eqs,
 \end{equation*}
where $d=m_1+m_2$ and
\begin{equation*}
 \sigma^2_Z \eqdef \max \left\{\left\| \frac{1}{n} \sum_{i=1}^{n} \EE[  \Bern{Z_{i}}\Bern{Z_{i}^\top} ] \right\|_{\sigma,\infty},
 \left\| \frac{1}{n}\sum_{i=1}^{n} \EE[ \Bern{Z_{i}}^\top\Bern{Z_{i}} ]\right\|_{\sigma,\infty}\right\} \eqs.
\end{equation*}
 In particular it implies that with at least probability $1-\rme^{-t}$
 \begin{equation*}
  \norm{\frac{1}{n} \sum_{i=1}^{n} \Bern{Z_{i}}}_{\sigma,\infty} \leq c^* \max \left\{ \sigma_Z \sqrt{ \frac{t + \log(d)}{n}} ,  \frac{U(t + \log(d))}{3n} \right \} \eqs,
 \end{equation*}
with $c^*=1+\sqrt{3}$.
\end{proposition}
\begin{proof}
The first claim of the proposition is Bernstein's inequality for random matrices (see for example
\cite[Theorem 1.6]{Tropp12}).
Solving the equation (in $t$) $- \frac{nt^2/2}{\sigma^2_Z+Ut/3} + \log(d)=-v$ gives with at least probability $1-\rme^{-v}$
 \begin{equation*}
  \norm{\frac{1}{n} \sum_{i=1}^{n} \Bern{Z_{i}}}_{\sigma,\infty} \leq  \frac{1}{n} \left[\frac{U}{3}(v + \log(d))+\sqrt{\frac{U^2}{9}(v + \log(d))^2+2n\sigma_Z^2(v + \log(d))}\right]\eqs,
 \end{equation*}
 we conclude the proof by distinguishing the two cases $n\sigma_Z^2 \leq (U^2/9)(v + \log(d))$ or $n\sigma_Z^2 > (U^2/9)(v + \log(d))$.
\end{proof}

\begin{lemma}
 \label{lem:MatExp}
Let $h \geq 1$. With the same assumptions as \Cref{prop:bernstein},
assume $n \geq (U^2\log(d))/(9\sigma^2_Z)$ then the following holds:
\begin{equation*}
  \EE \left[\norm{\frac{1}{n} \sum_{i=1}^{n} \Bern{Z_{i}}}_{\sigma,\infty}^h \right] \leq \left(\frac{2ehc^{*2}\sigma_Z^2\log(d)}{n}\right)^{h/2} \eqs,
\end{equation*}
with $c^*=1+\sqrt{3}$.
\end{lemma}

\begin{proof}
The proof is adapted from~\cite[Lemma 6]{Klopp14}.
Define $t^*\eqdef (9n\sigma_Z^2)/U^2-\log(d)$ the value of $t$ for which the two bounds of \Cref{prop:bernstein} are equal.
Let $\nu_1\eqdef n/(\sigma_Z^2c^{*2})$ and $\nu_2\eqdef3n/(U c^*)$ then, from \Cref{prop:bernstein} we have
\begin{align*}
\PP \left( \norm{\frac{1}{n} \sum_{i=1}^{n} \Bern{Z_{i}}}_{\sigma,\infty} > t \right) &\leq d \exp(-\nu_1t^2) \text{ for } t \leq t^* \eqs, \\
\PP \left( \norm{\frac{1}{n} \sum_{i=1}^{n} \Bern{Z_{i}}}_{\sigma,\infty} > t \right) &\leq d\exp(-\nu_2t) \text{ for } t \geq t^* \eqs,
\end{align*}
Let $h\geq 1$, then
\begin{align*}
 &\EE \left[\norm{\frac{1}{n} \sum_{i=1}^{n} \Bern{Z_{i}}}_{\sigma,\infty}^h \right] \leq \EE \left[\norm{\frac{1}{n} \sum_{i=1}^{n} \Bern{Z_{i}}}_{\sigma,\infty}^{2h\log(d)} \right]^{1/(2\log(d))} \eqs, \\
 &\leq \left( \int_0^{+\infty} \PP \left( \norm{\frac{1}{n} \sum_{i=1}^{n} \Bern{Z_{i}}}_{\sigma,\infty} > t^{1/(2h\log(d))} \right)dt \right)^{1/(2\log(d))}\eqs, \\
 & \leq d^{(2h\log(d))^{-1}} \! \left( \int_0^{+\infty}\!\!\!\! \exp(-\nu_1t^{2/(2h\log(d))}) + \exp(-\nu_2t^{1/(2h\log(d))}) dt  \right)^{1/(2\log(d))} ,\\
 &\leq  \sqrt{e}\!  \left(h\log(d) \nu_1^{-h\log(d)}\Gamma(h\log(d)) +2h\log(d) \nu_2^{-2h\log(d)}\Gamma(2h\log(d))  \right)^{1/(2\log(d))} ,
\end{align*}
where we used Jensen's inequality for the first line. Since Gamma-function satisfies for $x\geq 2$, $\Gamma(x)\leq (\frac{x}{2})^{x-1}$ (see \cite[Proposition 12]{Klopp11b}) we have
\begin{align*}
 &\EE \left[\norm{\frac{1}{n} \sum_{i=1}^{n} \Bern{Z_{i}}}_{\sigma,\infty}^h \right]\leq  \\
 & \sqrt{e}\! \left( (h\log(d))^{h\log(d)}\nu_1^{-h\log(d)}2^{1-h\log(d)}\! +\! 2(h\log(d))^{2h\log(d)}\nu_2^{-2h\log(d)}\right)^{1/(2\log(d))}.
\end{align*}
For $n \geq (U^2\log(d))/(9\sigma^2_Z)$ we have $\nu_1 \log(d) \leq \nu_2^2$ and therefore we get
\begin{equation*}
 \EE \left[\norm{\frac{1}{n} \sum_{i=1}^{n} \Bern{Z_{i}}}_{\sigma,\infty}^h \right] \leq \left(\frac{2eh\log(d)}{\nu_1}\right)^{h/2} \eqs.
\end{equation*}
\end{proof}

\subsection{Proof of \Cref{th:lower}}
\label{proof:thlower}
\begin{proof}
Let $h$ be the following function
\begin{equation}
\label{equation_ae}
h(\kappa)=\min\left \{1/2, \sqrt{\alpha\,r M K^{-1}_{(1-\kappa)\gamma}}/(8\gamma \sqrt{n}) \right \} \eqsp.
\end{equation}
Since $0 < h(\kappa) \leq 1/2$ and $h$ is continuous,
there exists a fixed point $\kappa^{*} \in (0,1/2]$:
\begin{equation}
\label{eq:definition-kappa-*}
h(\kappa_*)= \kappa_* \eqsp.
\end{equation}
For notational convenience, the dependence of $\kappa_*$ in $r, M$ and $n$ is implicit.
We start with a packing set construction, inspired by \cite{Davenport_Plan_VandenBerg_Wootters12}.
Assume \wlg that $m_1\geq m_2$. For $\kappa\leq 1$, define
$$
\mathcal{L} \, =\left \{ L = (l_{ij})\in\RR^{m_1\times r}:
l_{ij}\in\left \{- \frac{\kappa\gamma}{2}, \frac{\kappa\gamma}{2}\right \}\,, \forall i \in [m_1],\, \forall j \in [r] \right \},
$$
and consider the associated set of block matrices
$$
\mathcal{L}' \ =\ \Big\{
L'=(\begin{array}{c|c|c|c}L&\cdots&L&O
\end{array})\in\matset{m_1}{m_2}: L\in \mathcal{L}\Big\},
$$
where $O$ denotes the $m_1\times (m_2-r\lfloor m_2/r \rfloor )$ zero
matrix, and $\lfloor x \rfloor$ is the integer part of $x$.

\begin{remark}
In the case $m_1< m_2$, we only need to change the construction of the low rank component of the test set.
We first build a matrix $\tilde L  \in \mathbb R^{r \times m_2} $  with
entries in $\left \{- \frac{\kappa\gamma}{2}, \frac{\kappa\gamma}{2}\right \}$ and
then we replicate this matrix to obtain a block matrix $L$ of size $m_1 \times m_2$.
\end{remark}
  Let $\mathbf{I}_{m_1\times m_2}$ denote the $m_1\times m_2$ matrix of ones.
  The Varshamov-Gilbert bound (\cite[Lemma 2.9]{Tsybakov09}) guarantees the existence of a subset $\mathcal{L}''\subset\mathcal{L}'$ with
cardinality $\mathrm{Card}(\mathcal{L}'') \geq 2^{(rM)/8}+1$ containing the matrix $(\kappa\gamma/2)\, \mathbf{I}_{m_1\times m_2}$ and such that, for any two
distinct elements $\mat{X_1}$ and $\mat{X_2}$ of $\mathcal{L}''$,
\begin{equation} \label{lower_2}
\Arrowvert \mat{X_1}-\mat{X_2}\Arrowvert_{2}^2  \geq \frac{Mr\,\kappa^2\gamma^{2}}{8}
\left\lfloor \frac{m_2}{r}\right\rfloor \geq
\frac{m_1m_2\,\kappa^2\gamma^{2}}{16}\,.
\end{equation}
Then, we construct the packing set $\mathcal{A}$ by setting
$$\mathcal{A}=\left \{L+\frac{(2-\kappa)\gamma}{2}\,\mathbf{I}_{m_1\times m_2}\;:\;L\in \mathcal{L}''\right \}.$$
By construction, any element of $\mathcal{A}$ as well
as the difference of any two elements of $\mathcal{A}$ has rank at most $r$, the entries of any matrix in
$\mathcal{A}$ take values in $[0,\gamma]$, and $X^0 = \gamma \mathbf{I}_{m_1\times m_2}$ belongs to $\mathcal{A}$. Thus,
$\mathcal{A}\subset {\mathcal F}(r,\gamma)$. Note that $\mathcal{A}$ has the same size as $\mathcal{L}''$
and it also satisfies the same bound on pairwise distances, i.e. for any two
distinct elements $\mat{X_1}$ and $\mat{X_2}$ of $\mathcal{A}$, \eqref{lower_2} is satisfied.

For some $\mat{X} \in \AA$, we now estimate  the Kullback-Leibler
divergence $\KLd{\PP_{\mat{X^0}}\!}{\PP_{\mat{X}}\!}$ between probability measures $\PP_{\mat{X^0}}$
and $\PP_{\mat{X}}$. By independence of the observations $(Y_i,\omega_i)_{i=1}^n$,
\begin{equation*}
 \KLd{\PP_{\mat{X^0}}}{\PP_{\mat{X}}}=n\EE_{\omega_1} \left[ \sum_{\class=1}^2 f^\class( \mat{X^0_{\omega_1}}) \log \left(\frac{f^\class( \mat{X^0_{\omega_1}})}{f^\class( \mat{X_{\omega_1}})}\right) \right] \eqsp.
\end{equation*}
Since $\mat{X^0_{\omega_1}}= \gamma$ and
either $\mat{X_{\omega_1}}=\mat{X^0_{\omega_1}}$ or $\mat{X_{\omega_1}}=(1-\kappa)\gamma$,
by \Cref{lem:kul_simple} we get
\begin{equation*} 
\KLd{\PP_{\mat{X^0}}}{\PP_{\mat{X}}}\leq\frac{n\left [f^1(\gamma)-f^1\left ((1-\kappa)\gamma\right )\right ]^{2}}{f^1\left ((1-\kappa)\gamma\right )\left [1-f^1\left ((1-\kappa)\gamma\right )\right ]}.
\end{equation*}
From the mean value theorem, for some $\xi\in [(1-\kappa)\gamma,\gamma]$ we have
\begin{equation*} 
\KLd{\PP_{\mat{X^0}}}{\PP_{\mat{X}}}\leq\frac{n\{(f^1)'(\xi)\}^{2} (\kappa\gamma)^{2}}{f^1\left ((1-\kappa)\gamma\right )\left [1-f^1\left ((1-\kappa)\gamma\right )\right ]}\eqsp.
\end{equation*}
Using \Cref{h:minimax}, the function $(f^1)'$ is decreasing and the latter inequality implies
\begin{equation} \label{KLdiv}
\KLd{\PP_{\mat{X^0}}}{\PP_{\mat{X}}}\leq 8\,n (\kappa\gamma)^{2} g((1-\kappa)\gamma) \eqsp,
\end{equation}
where $g$ is defined in \eqref{eq:definition:K-gamma-bin}.
From \eqref{KLdiv} and plugging $\kappa= \kappa^*$ defined in \cref{eq:definition-kappa-*}, we get
\[
\KLd{\PP_{\mat{X^0}}}{\PP_{\mat{X}}} \leq \frac{\alpha r M}{8} \leq \alpha \log_2(rM/8)  \eqsp,
\]
which implies that
\begin{equation}\label{eq: condition C}
\frac{1}{\mathrm{Card}(\AA)-1} \sum_{\mat{X} \in \AA}\KLd{\PP_{\mat{X^0}}}{\PP_{\mat{X}}}\ \leq\ \alpha \log \big(\mathrm{Card}(\AA)-1\big) \eqsp.
\end{equation}
Using \eqref{lower_2} and \eqref{eq: condition C}, \cite[Theorem 2.5]{Tsybakov09} implies
\begin{equation}\label{lower_1}
\inf_{\mat{\hat{X}}}
\sup_{\substack{\mat{\tX}\in\,{\cal F}( r,\gamma)
}}
\mathbb P_{\mat{\tX}}\left (\dfrac{\Vert \mat{\hat{X}}-\mat{\tX}\Vert_2^{2}}{m_1m_2} > c\min\left \{\gamma^{2}, \dfrac{Mr}{n\,K_{(1-\kappa^{*})\gamma}}\right \} \right )\ \geq\ \delta
\end{equation}
for some  universal constants $c>0$ and $\delta\in(0,1)$.
\end{proof}

\begin{lemma}
\label{lem:kul_simple}
Let us consider $x,y\in(0,1)$ and
\begin{equation*}
 k(x,y) \eqdef x \log(x/y)+(1-x) \log((1-x)/{1-y}) \eqs.
\end{equation*}
Then the following holds
\begin{equation*}
k(x,y)  \leq \frac{(x-y)^2}{y(1-y)} \eqsp.
\end{equation*}
\end{lemma}

\begin{proof}
The proof is taken from \cite[Lemma 4]{Davenport_Plan_VandenBerg_Wootters12}. Since  $k(x,y) = k(1-x,1-y) $, \wlg we may assume $y>x$. The function $g(t)=k(x,x+t) $ satisfies $g'(t)=t/[(x+t)(1-x-t)]$ and
$g''(t) \geq 0$. Therefore the mean value Theorem
gives $g(y-x)-g(0)\leq g'(y-x)(y-x)$ which yields the result.
\end{proof}

\bibliographystyle{plain}
\bibliography{references_all}

\end{document}